\documentclass[12pt,reqno,draft,a4paper]{article}

\usepackage{hyperref}
\usepackage{algpseudocode}
\usepackage{algorithm}
\usepackage{amsmath}
\usepackage{amsthm}
\usepackage{amssymb}
\usepackage{dsfont}
\usepackage{amsfonts}
\usepackage{stmaryrd}
\usepackage{commath}
\usepackage{bm}
\usepackage{mathrsfs}
\usepackage{amstext}
\usepackage[shortlabels]{enumitem}
\usepackage{graphicx}

\newcommand{\R}{\field R}

\newcommand{\N}{\field N}
\newcommand{\Sphere}{\ensuremath{\mathds S^{d-1}}}

\newcommand{\prob}{\ensuremath{\mathbb{P}}}

\newcommand{\expec}{\ensuremath{\mathds E}}

\newcommand{\error}{\ensuremath{\mathrm{err}}}

\newcommand{\ranerror}{\ensuremath{\mathrm{err}^{\mathrm{ran}}}}

\newcommand{\tr}{{\scriptscriptstyle T}}
\newcommand{\dd}{\mathrm{d}}

\newcommand{\Lip}{\ensuremath{\mathrm{Lip}}}

\renewcommand{\R}{\mathbb{R}}
\renewcommand{\N}{\mathbb{N}}

\DeclareMathOperator{\sign}{sgn}

\newcommand{\HIT}{\text{\rm HIT}}

\theoremstyle{plain}
\newtheorem{theorem}{Theorem}
\newtheorem{lemma}[theorem]{Lemma}
\newtheorem{corollary}[theorem]{Corollary}

\theoremstyle{definition}
\newtheorem{defi}[theorem]{Definition}
\newtheorem{remark}[theorem]{Remark}

\title{
The recovery of ridge functions on the hypercube suffers from the curse of dimensionality
}

\author{Benjamin Doerr\footnote{
Laboratoire d'Informatique (LIX), \'Ecole Polytechnique,
CS35003, 91120 Palaiseau, France
email: doerr@lix.polytechnique.fr}
\hspace*{24pt}
Sebastian Mayer\footnote{Corresponding author,
Fraunhofer Center for Machine Learning and Fraunhofer-Institute for Algorithms and Scientific Computing SCAI,
Schloss Birlinghoven, 53754 Sankt Augustin, Germany,
email: sebastian.mayer@scai.fraunhofer.de}}

\begin{document}

\maketitle

\begin{abstract}
A multivariate ridge function is a function of the form~$f(x) = g(a^\tr x)$, where~$g$ is univariate and~$a \in \R^d$.
We show that the recovery of an unknown ridge function defined on the hypercube $[-1,1]^d$ with Lipschitz-regular profile $g$ suffers from the curse of dimensionality when the recovery error is measured in the $L_\infty$-norm, even if we allow randomized algorithms. If a limited number of components of $a$ is substantially larger than the others, then the curse of dimensionality is not present and the problem is weakly tractable provided the profile $g$ is sufficiently regular.
\end{abstract}

\section{Introduction}

In the \emph{uniform recovery problem} (or \emph{$L_\infty$-recovery problem}), 
the aim is to compute an approximation~$\widehat f$ of an 
unknown function~$f: D \subseteq \R^d \to \R$ 
such that the approximation error~$\|f - \widehat f\|_\infty$ 
is small. The only available information about~$f$ is a sequence of samples~$f(x_1), 
\dots, f(x_n)$ and that~$f$ belongs to some class of functions~$F_d$, which describes the a priori model assumptions. The sampling points~$x_1, \dots, x_n$ 
may be freely chosen. A 
measure for the difficulty of the recovery problem is the 
so-called \emph{information complexity}~$n(\varepsilon, F_d)$. 
It is the smallest number $n$ such that there is an algorithm
which evaluates at most $n$ samples and achieves an error~$\|f- \widehat{f}\|_\infty \leq \varepsilon$, irrespective of 
which~$f \in F_d$ 
is presented as input to the algorithm. The general 
question is what properties of the function class~$F_d$ 
make the recovery problem efficiently solvable.

For functions depending only on a few variables, 
regularity has proved to be sufficient for the existence 
of efficient algorithms. 
This is nicely demonstrated by the classic notion 
of \emph{generalized Lipschitz regularity}. For~$r > 0$ 
and~$m = \lceil r-1 \rceil$, consider the following class of 
univariate functions~$B^{\Lip(r)}$ defined on~$[-1,1]$. 
Every~$g \in B^{\Lip(r)}$ is~$m$-times 
continuously differentiable, we have $\max\{\|g\|_\infty, \|g^{(1)}\|_\infty,\dots,\|g^{(m)}\|_\infty\} \leq 1$,
and the~$m$-th derivative is H\"older continuous with 
exponent~$\beta = r - m$. 
Then, it is well-known that the worst-case approximation error of the optimal algorithm decays polynomially in the number of samples,
\begin{align}\label{eq:err_univariate_lipschitz} 
 c_r \, n^{-r} \leq \error(n,B^{\Lip(r)}) \leq C_r \, n^{-r},
\end{align}
with positive constants~$c_r$,~$C_r$ 
depending only on~$r$. The optimal algorithm is given by a spline. Since~$\error(n,B^{\Lip(r)})$ 
is inverse to the information complexity~$n(\varepsilon,B^{\Lip(r)})$, 
this implies~$n(\varepsilon,B^{\Lip(r)}) \simeq \varepsilon^{-1/r}$.

Using the same notion of regularity for~$d$-variate functions, 
the picture changes dramatically for large $d$. 
Let~$B^{\Lip(r)}_d$ be the counterpart\footnote{See, e.g., the monograph~\cite{devore:constructive_approximation} for a formal definition.} 
of~$B^{\Lip(r)}$ for~$d$-variate functions defined on the cube~$[-1,1]^d$.
It is a classical result from approximation theory~\cite{bakhvalov:2015:int} that
\[ 
 c_{r,d} \, n^{-r/d} \leq \error(n,B^{\Lip(r)}_d) 
 \leq C_{r,d} \, n^{-r/d},
\]
where~$c_{r,d},C_{r,d}$ denote positive constants depending on~$r$ and~$d$.
This shows that the asymptotic decay of the error is extremely slow in large dimensions and that for small error thresholds $\varepsilon$, we certainly have $n(\varepsilon,B^{\Lip(r)}_d) \simeq_d (1/\varepsilon)^{d/r}$. That indeed any algorithm needs exponentially many samples to guarantee a non-trivial error has been shown only recently by Novak and Wo\'zniakowski~\cite{novak/woz:2009:infinite_intractable}, who proved that 
\[ 
 n(\varepsilon,B^{\Lip(r)}_d) \geq 2^{\lfloor d/2 \rfloor} 
\]
for all $\varepsilon \in (0,1)$ and $d \in \N$.
Hence, this recovery problem is \emph{intractable} and 
suffers from the \emph{curse of dimensionality} in the strict sense of Information-based Complexity (IBC). 

The considerations made so far clearly demonstrate that 
if we want the uniform recovery problem to be efficiently
solvable in high dimensions, then we need a priori assumptions
stronger than just regularity. In this paper, we study the
assumption that the unknown function is a \emph{ridge function}
\begin{align}\label{eq:ridge_function} 
 f(x) = g(a^\tr x ), \quad \text{ with } g \in  B^{\Lip(r)} \text{ and } \|a\|_1 \leq 1.
\end{align}
It is common to call the univariate function $g$ the ridge functions's \emph{profile}
and to call the~$d$-dimensional vector~$a$ the \emph{ridge vector}.
Like a linear function, a ridge function is constant along hyperplanes and so we hope that this prior knowledge greatly reduces the complexity of the recovery problem. This idea is not new. In statistics, models based on ridge functions
have been used since the early 1980s to avoid the typical issues occurring in nonparametric regression problems
over high-dimensional domains. We give a more detailed overview of research on ridge functions in Section \ref{sec:related_work}.
In the context of the uniform recovery problem, ridge functions have first been studied by Cohen et al.~\cite{cohen:capturing_ridge}. 
Additionally to~\eqref{eq:ridge_function}, they assumed that
\begin{align}\label{eq:known_signs}
a_i \geq 0, \quad i=1,\dots,d,
\end{align}
and that, for some $0< p \leq 1$ and $S \in \{1,\dots, d-1\}$,
\begin{align}\label{eq:approximate_sparsity}
\|a\|_p \leq 1 \quad \text{ and } \quad \|a\|_1 \geq \min\{1, 4S^{1-1/p}\}.
\end{align}
Assumption \eqref{eq:known_signs} is equivalent to knowing the signs of the ridge vector's components in advance. If $0<p<1$, assumption \eqref{eq:approximate_sparsity} implies that roughly $S$ components of the ridge vector have to be substantially larger than the others\footnote{Strictly speaking, the paper~\cite{cohen:capturing_ridge} assumes that $\|a\|=1$ and $\|a\|_p \leq M$ for some positive constant $M$. This is equivalent to assuming \eqref{eq:approximate_sparsity}. The latter formulation is more convenient for our considerations.}. We call this \emph{approximate sparsity} and note that it is a stronger condition than \emph{compressibility}, which only asks for~$\|a\|_p \leq 1$, see~\cite[p.~42]{rauhut/foucart:compressive_sensing}. Given a ridge function $f$ such that \eqref{eq:ridge_function}, \eqref{eq:known_signs}, and \eqref{eq:approximate_sparsity} are fulfilled, Cohen et al.~\cite{cohen:capturing_ridge} employ spline approximation and compressive sensing techniques~\cite{rauhut/foucart:compressive_sensing} to obtain an approximation $\widehat f$ with error bound
\begin{align}\label{eq:cohen_ub}
\begin{split}
\|f - \widehat{f}\|_\infty \lesssim n^{-r} + 
\begin{cases}
\left( \frac{\log(ed/n)}{n}\right)^{1/p - 1} &, n < d,\\
0 &, n \geq d.
\end{cases}
\end{split}
\end{align}
So under the given conditions, the recovery of a multivariate ridge function is polynomially tractable and almost as easy as the univariate problem, see~\eqref{eq:err_univariate_lipschitz}.

Assumption~\eqref{eq:known_signs} is rather restrictive as it does not allow to model situations where some of the variables~$x_1,\dots,x_d$ may have inhibitory effects but it is not clear which ones. Hence, we investigate in this paper consequences for the complexity if we drop assumption~\eqref{eq:known_signs} and allow ridge vectors with negative entries, that is, we study the recovery of ridge functions from the class
\begin{align}\label{eq:class_approximately_sparse}
 R_d^{r,(p,S)} = \left\{ f: [-1,1]^d \to \R: f(x) = g(a^\tr x), \; g \in B^{\Lip(r)}, \; a \text{ fulfills } \eqref{eq:approximate_sparsity} \right\},
\end{align}
where $r>0$, $0<p \leq 1$, and $S \in \{1,\dots,d-1\}$. Moreover, we allow algorithms to use randomness, e.g., to use random sampling points. The quantity of interest, for which we wish to prove lower and upper bounds, is then the $n$th minimal worst-case error in the randomized setting,
\[ 
 \ranerror(n, R_d^{r,(p,S)}) = \inf_{S_n} \sup_{f \in R_d^{r,(p,S)}} \left(\expec[\|f - S_n(f)\|_\infty^2]\right)^{1/2},
\]
where the infimum is taken over all admissible randomized algorithms using at most $n$ function evaluations. See Section~\ref{sec:def} for a formal definition of the randomized setting. Note that
\[ 
 \ranerror(n, R_d^{r,(p,S)}) \leq \error(n, R_d^{r,(p,S)}). 
\]

In turns out that dropping assumption~\eqref{eq:known_signs} leads to a drastic change in the complexity. For~$p=1$, i.e., if ridge vectors are not approximately sparse, we show that
\[ 
 \ranerror(n, R_d^{r,(1,S)}) \gtrsim 1
\]
as long as $1 \leq n \lesssim e^{d/8}$, with an equivalence constant in the estimate that depends only on $r$. We conclude that the recovery of an unknown ridge function from the class~$R_d^{r,(1,S)}$ suffers from the curse of dimensionality, even if we allow sampling points to be chosen adaptively and at random.

When~$0<p<1$, the answer that we can give is not final.  We show the upper bound
\[ 
 \error(n,R_d^{r,(p,S)}) \leq C_{r,p,S}
 \begin{cases}
   1 &, 1 \leq n \leq d,\\
   \left(\frac{1}{\log(n)}\right)^{r(1/p-1)} &, d \leq n \leq 2^d d^{1/p-1},\\
   2^{rd} \, n^{-r} &, n \geq 2^d d^{1/p-1}
 \end{cases}
\]
for $r > 1$, $0<p<1$, and $S \in \{1,\dots,d-1\}$, see Theorem~\ref{chap:ridge:sec:cube:res:ub_det}. The algorithm establishing the upper bound is an extension of the algorithm used in~\cite{cohen:capturing_ridge}, which we have augmented by a search for the most important signs of the ridge vector to compensate for the dropped assumption~\eqref{eq:known_signs}. Although the extended algorithm reaches asymptotically an error decay of $n^{-r}$, it is important to note that we can establish this rate only for exponentially many sampling points~$n > 2^d d^{1/p-1}$. In the \emph{preasymptotic range}, we can only guarantee an error decay that is logarithmic in the number of samples.  This implies that the recovery of an unknown ridge function with approximately sparse ridge vector is at least \emph{weakly tractable}, provided
\[ 
 r > \frac{1}{1/p-1}.
\]
Unfortunately, it is unclear whether the constructed algorithm is optimal. We are only able to prove a lower bound for a different function class, namely 
\begin{align}\label{eq:class_compressible}
 R_d^{r,p} = \left\{ f: [-1,1]^d \to \R: f(x) = g(a^\tr x), \; g \in B^{\Lip(r)}, \; \|a\|_p = 1 \right\},
\end{align}
where $r > 0$, and $0< p \leq 1$.
For this class, we obtain the lower bound 
\[ 
 \ranerror(n, R_d^{r,p}) \gtrsim \left(\frac{1}{\log(2n)}\right)^{1/p-1},
\]
for $r>0$ and $0<p\leq 1$, provided $n \lesssim \exp(d/8)$, see Theorem~\ref{chap:ridge:sec:lb:res:lb_cube_ran}. Note that only for $p=1$, we have
$$
 R_d^{r,1} = R_d^{r,(1,S)}
$$
for all $r>0$ and $S \in \{1,\dots,d-1\}$. Otherwise, the function classes are different so that it remains an open problem to prove lower bounds for $\ranerror(n,R_d^{r,(p,S)})$ when $0<p<1$.
 
\paragraph{Outline.} 
The paper is organized as follows. 
We begin with a thorough 
definition of the complexity-theoretical setup in Section~\ref{sec:def}, where we give a definition what we consider to be a deterministic and a randomized algorithm. Then, in Section \ref{sec:lower}, 
we prove lower bounds for the worst-case error
for both deterministic and randomized algorithms. 
Section \ref{sec:algo} is dedicated to the description of our algorithm, followed by a detailed error analysis in Section~\ref{sec:error}, which leads to upper bounds on the worst-case error.
Finally, 
we discuss related work in Section~\ref{sec:related_work}, in particular, relations to 
the regression problem in semiparametric statistics.

\paragraph{Acknowledgement.}
This work was partially developed in the Fraunhofer Cluster of Excellence ``Cognitive Internet Technologies''.

\section{Preliminaries}
\label{sec:def}

As usual, we denote by $\N$ the natural numbers $1,2,3,\dots$. Throughout this paper, let~$d \in \N$. For sequences $(f_n)_{n \in \N}$, $(g_n)_n{n \in \N}$, we write $f_n \lesssim g_n$ whenever there is a constant $C > 0$ such that $f_n \leq C g_n$ for all $n \in \N$. Note that the constant need not be absolute. Where necessary, we indicate on what parameters the constant depends.

Let us recall some basic notions. For~$p>0$ and~$x \in \R^d$, the quasi-norm~$\|x\|_p$ is given by
\[ 
 \|x\|_p := \big(\sum_{i=1}^d |x_i|^p\big)^{1/p}.
\]
For $D \subset \R^d$, consider a function $f:D \to \R$. The \emph{uniform norm} $\|f\|_\infty$ is given by
\[ 
 \|f\|_\infty := \sup_{x \in D} |f(x)|.
\]
For any positive number~$0<\beta\leq1$, the \emph{H\"older constant} of order~$\beta$ is given by
\begin{equation}\label{chap:fs:eq:hoelder_defi}
|f|_{\beta}:=\sup_{\substack{x,y\in[-1,1]^d\\ x\not=y}}\frac{|f(x)-f(y)|}{2\min\{1, \|x-y\|_1\}^{\beta}}\;.
\end{equation}
We say that~$f$ is \emph{H\"older-continuous} of order~$\beta$ if~$|f|_\beta < \infty$. This definition immediately implies the relation 
\begin{align}\label{chap:fs:eq:hoelder_ineq}
  |f|_{\beta} \leq |f|_{\beta'}\mbox{ if } 0<\beta < \beta' \leq 1.
\end{align}
Let $C([-1,1]^d)$ be the space of continuous functions defined on $[-1,1]^d$, equipped with the norm~$\|\cdot\|_\infty$.
\subsection{Deterministic algorithms}

Understanding the worst-case complexity of the uniform recovery problem for a given function class~$F_d$ means to understand how \emph{any} possible algorithm performs in the worst-case on the given class. This requires a rigorous definition of what we consider to be a feasible algorithm. The field of Information-based Complexity (IBC) provides a well-established framework that we follow in this work. Let us first consider the \emph{deterministic setting}, where algorithms are only allowed to acquire information about a function in a deterministic, i.e. non-random, fashion. Note that since all ridge functions in the classes~\eqref{eq:class_approximately_sparse} and~\eqref{eq:class_compressible} are continuous, it is sufficient to define the concept of algorithm with respect to a general function class $F_d \subset C([-1,1]^d)$. Moreover, we restrict our considerations to algorithms that only use function evaluations as information operations and not more general linear functionals. For a general definition of the deterministic setting, discussions and references, we refer to \cite{novak/woz:tractability_volI}.

A deterministic algorithm using at most $n$ function evaluations is a mapping~$S_n$ that maps a function $f \in F_d$ to an approximant $S_n(f) \in C([-1,1]^d)$. More specifically, the approximant is given by
$$
 S_n(f) = \phi(f(x_1),\dots, f(x_n)),
$$
where $fx_1,\dots,x_n \in [-1,1]^d$ are sampling points and $\phi: \R^n \to C([-1,1]^d)$. The first sampling point $x_1$ is completely independent of the input $f$, while the choice of the remaining points can be adaptive, that is,~$x_i$ may functionally depend on the function values~$f(x_1), \ldots, f(x_{i-1})$. Formally, this means that there are functions
$$
\psi_i\colon \R^{i-1}\to 
[-1,1]^d, \quad i = 2, \ldots, n
$$
that recursively define the sampling points via
$$
x_i = \psi_i(f(x_1), \ldots, f(x_{i-1})).
$$
We do not make any computational assumptions, e.g., that the functions~$\phi$ and~$\psi_i$ describing the algorithm are efficiently computable in some specific model of computation. Beside the a priori information that~$f$ is in the class~$F_d$, the algorithm~$S_n$ has no other information than the function values~$f(x_1), \ldots, f(x_n)$.

It remains to define precisely how we quantify the complexity of the recovery problem given a function class $F_d$. Let us first introduce the~\emph{$n$th minimal worst-case error}
\[ 
 \error(n, F_d) := \inf_{S_n} \sup_{f \in F_d} \|f - S_n(f)\|_\infty,
\]
where the infimum is taken over all deterministic algorithms that use at most $n$ function values.
Then, we define the \emph{information complexity} as the inverse of the minimal worst-case error,
\[ 
 n(\varepsilon, F_d) := \min\{ n \in \N: \error(n,F_d) \leq \varepsilon \}, \quad \varepsilon > 0.
\]

\begin{remark}
The information complexity neglects any computational cost. This is justified as in function recovery problems, the information cost are usually dominating. In particular, for the algorithm studied in Section~\ref{sec:algo}, the computational cost are proportional to the number of used function samples.
\end{remark}

\subsection{Randomized algorithms}

In this paper, we also wish to study algorithms that use randomness in the choice of the sampling points~$x_1,\dots,x_n$ and the mapping $\phi$. While there is a clear agreement in IBC what to consider a deterministic algorithm, the situation is less settled when it comes to randomized algorithms. The crux are measurability assumptions. We follow the common approach in IBC and assume just as much measurability as required in our proofs. As a result, our definition of randomized algorithm will be less general than in~\cite{novak/woz:tractability_volI}, but closely resemble~\cite{heinrich:lower_bounds_mc_approxmation}.

We first give a precise definition of what we consider to be a sequence of adaptively chosen random sampling points.

\begin{defi}\label{chap:ibc:defi:randomized_information}
Let~$n \in \N$. A sequence of $n$ adaptively chosen sampling points is a sequence of random variables $X_1, \dots, X_n$ defined over a common probability space~$(\Omega,\mathfrak A, \prob)$ that take values in $[-1,1]^d$ and fulfill the following. For every~$i=2,\dots,n$, there is a mapping
\begin{align*}
 L_i &: \Omega \times \R^{i-1} \to [-1,1]^d
\end{align*}
such that, for all $\omega \in \Omega$,
  \[ 
   (x_1,\dots,x_{i-1}) \mapsto L_i(\omega,x_1,\dots,x_{i-1})
  \]
is Borel measurable and
\[ 
 X_i(\omega) = L_i(\omega, f(X_1(\omega)), \dots, f(X_{i-1}(\omega))).
\]
\end{defi}

\begin{defi}\label{chap:ibc:defi:randomized_algorithm}\index{randomized algorithm}
A randomized algorithm using at most~$n$ information operations is given by a probability space~$(\Omega, \mathfrak A, \prob)$ and a mapping~$$S_n: \Omega \times F \to C([-1,1]^d)$$ such that
$$S_n(\omega,f) = \phi(\omega, f(X_1(\omega)), \dots, f(X_n(\omega))),$$
where
\begin{itemize}
\item $\phi: \Omega \times \R^n \to C([-1,1])$ is measurable in the first argument w.r.t. $\mathfrak{A}$ and Borel measurable in the last $n$ arguments,
\item $X_1,\dots,X_n$ is a sequence of adaptively chosen random sampling points according to Definition~\ref{chap:ibc:defi:randomized_information}.
\end{itemize}
\end{defi}

\noindent It remains to define the~$n$th \emph{minimal worst-case error} in the randomized setting as
\[ 
 \ranerror(n, F_d) := \inf_{S_n} \sup_{f \in F_d} \left(\expec[\|f - S_n(f)\|_\infty^2]\right)^{1/2}, 
\]
where the infimum is taken over all admissible randomized algorithms using at most $n$ function evaluations. The information complexity in the randomized setting is given by
$$n^{\textrm{ran}}(\varepsilon,F_d) = \inf\{n \in \N: \ranerror(n,F_d) \leq \varepsilon\}, \quad \varepsilon > 0.$$


\subsection{Complexity classes}
In IBC research, various complexity classes for continuous problems have been introduced. Let us introduce those that we encounter in this work. A problem is said to \emph{suffer from the curse of dimensionality} in the deterministic setting
if there are~$C>0$ and~$\gamma>1$ such that 
\[
 n(\varepsilon, F_d) \geq C \gamma^d
\]
holds for all~$\varepsilon>0$ and infinitely many $d \in \N$. Furthermore, a problem is said to be \emph{weakly tractable} if
\[
 \lim_{\varepsilon^{-1}+d\to \infty} 
 \frac{\log n (\varepsilon, F) }{\varepsilon^{-1}+d} = 0.
\]
Finally, a problem is \emph{polynomially tractable} if there are $C,p,q > 0$ sucht that for all~$\varepsilon > 0$ and all~$d \in \N$, we have
\[ 
 n(\varepsilon, F_d) \leq C (1/\varepsilon)^p d^q.
\]
The same notions of tractability can be introduced in the randomized setting by replacing~$n(\varepsilon, F_d)$ by~$n^{\textrm{ran}}(\varepsilon, F)$ in the above definitions. For further levels of tractability, we refer to~\cite{gnewuch/woz:2011:quasi, novak/woz:tractability_volI,novak/woz:tractability_volII,novak/woz:tractability_volIII,siedlecki:uwt,siedlecki/weimar:st-wt}.

\subsection{Approximation of univariate Lipschitz functions}

Let~$r>0$ and $a<b$. By~$m=\llfloor r \rrfloor$ we denote the largest integer strictly
less than~$r$. The \emph{Lipschitz space}~$\Lip^r([a,b])$ is given by all
univariate functions~$g\colon [a,b]\to \R$ such that the \emph{Lipschitz norm}
\[
 \norm{g}_{{\rm Lip}(r)} = \max\left\{\norm{g}_\infty,\|g^{(1)}\|_{\infty},\dots,\,\|g^{(s)}\|_{\infty},|g^{(s)}|_\beta\right\}
\]
is finite,
where~$g^{(i)}$ is the~$i$th derivative and
\begin{align}\label{eq:Hoelder-continuity}
  |g^{(m)}|_\beta = \sup_{u,v \in [-1,1]} \frac{|g^{(s)}(u)-g^{(s)}(v)|}{2\min\{ 1,\abs{u-v}\}^\beta}, \quad \beta=r-m\in (0,1]
\end{align}
the\emph{H\"older constant}. By $B^{\Lip(r)}$ we denote the closed unit ball of $\Lip^r([-1,1])$.

We recapitulate some basic facts of univariate spline approximation, see~\cite[Chap. 12]{devore:constructive_approximation} for further background. For~$r_0,n \in \N$ and~$h=1/n$, let~$\mathcal{P}_h$ be the space
of piecewise polynomials of degree~$r_0-1$ over equidistant intervals, determined by the points
$ih$,~$i\in\{\pm 1,\dots,\pm n\}$, such that on each of these interval the piecewise polynomials have continuous derivatives of order~$r_0-2$. 
A \emph{quasi-interpolant}
$Q_h$ with step-size~$h$ is a linear operator mapping from the continuous functions on~$[-1,1]$ to~$\mathcal{P}_h$ such that the application of~$Q_h$ uses only the 
function values at the points
$$ih, \quad i\in\{\pm 1,\dots, \pm n\}.$$
The resulting spline has the following approximation properties.

\begin{lemma}\label{chap:ridge:lem:quasi-interpolation}
For~$0<r \leq r_0$ 
and 
$g \in B^{\Lip(r)}$
it holds that
\begin{equation} \label{chap:ridge:eq:spline_propQ1}
 \norm{g-Q_h g}_{L_\infty} \leq c_r h^r
\end{equation}
with a constant~$c_r$ depending only on~$r$, and 
\begin{equation} \label{chap:ridge:eq:spline_propQ2}
 \norm{Q_h g}_{L_\infty} \leq c_r \max_{i\in\{ \pm 1,\dots,\pm n_g \}} \abs{g(ih)},
\end{equation}
with again a constant~$c_r$ depending only on~$r$.
\end{lemma}

We also need extrapolation in our error analysis. That is, we need error estimates for points outside the interval that has been sampled. Although Lemma \ref{chap:ridge:lem:quasi-interpolation} does not directly apply then,  the reader familiar with spline approximation knows that~$Q_h g$ can also be used for extrapolation and that properties similar to~\eqref{chap:ridge:eq:spline_propQ1} and~\eqref{chap:ridge:eq:spline_propQ2} hold true. However, we could not find an explicit statement suitable for our needs in the literature. Hence, let us collect what we need later on in Section~\ref{sec:error} in terms of Taylor polynomials.

Given reals~$a < b$, let~$g \in \Lip^r([a,b])$ and consider the Taylor polynomial
\[ 
 T_{m,t_0}g(t) = g(t_0) + \sum_{i=1}^m \frac{g^{(i)}(t_0)}{i!} (t-t_0)^i,
\]
where~$m = \llfloor r \rrfloor$. Let~$\beta = r - m$.
For any~$t_0, t_1 \in [a,b]$, the standard error estimate for Taylor polynomials gives
\begin{align}\label{chap:ridge:eq:extrapolation_propQ1} 
 |g(t_1) - T_{m,t_0}g(t_1)| \leq \frac{2}{m!} |g^{(m)}|_\beta (t_1 - t_0)^r,
\end{align}
which is the required counterpart of~\eqref{chap:ridge:eq:spline_propQ1}.

As a counterpart to~\eqref{chap:ridge:eq:spline_propQ2}, we need that if the approximations of the derivatives $g^{(i)}$ are small in absolute value for all~$i=1,\dots,m$, then the polynomial~$|T_{m,t_0}g(\cdot)  - g(t_0)|$ has to be small in a neighborhood of $t_0$. To prove this, we have to consider divided differences. For~$m \in \N$, the~$m$th \emph{difference} with stepsize~$h \in \R$ in the point~$t \in \R$ of a univariate function~$g: \R \to \R$ is defined as
\[ 
 \Delta^m_h(g,t) := \sum_{j=0}^m {m \choose j} (-1)^{m-j} g(t + j h).
\]
The~$m$th \emph{divided difference} is given by~$D_h^m(g,t) := h^{-m} \Delta_h^m(g,t)$.
For our purposes, it is convenient to work with the representation 
\begin{align}\label{def:divided_difference}
D_h^m(g,t) = h^{-m+1}\sum_{j=0}^{m-1} {m-1 \choose j} (-1)^{m-1-j} D_h^1(g,t+jh),
\end{align}
which easily follows from the definition of~$\Delta_h^m(g,t)$. If~$g$ is~$m$ times continuously differentiable, then an iterative application of the mean value theorem of calculus gives
\begin{align}\label{eq:iterative_mean_value_thm}
 D_h^m(g,t)  = g^{(m)}(t+\xi), \quad \text{ for some } \xi \in [t,t+mh].
\end{align}

\begin{lemma}[Counterpart of~\eqref{chap:ridge:eq:spline_propQ2} for extrapolation]
\label{chap:ridge:lem:extrapolation_propQ2}
Let~$g \in B_{\Lip^r([a,b])}$ for~$r > 1$ and~$m = \llfloor r \rrfloor$. If, for some~$t_0 \in [a,b]$ and~$h > 0$, we have
\[
 \abs{D_{-h}^i(g,t_0)} \leq 2^{i-1}h^{r-i+1} \quad \text{for } i=1,\dots,m,
\]
then
\[ 
 g^{(i)}(t_0) \leq C_i h^{r-i} \qquad \qquad \text{for } i=1,\dots,m,
\]
with constants~$C_i \leq 2^m m!$. Consequently, for any~$t_1 \in [a,b]$,
\[ 
 |T_{m,t_0}g(t_1)-g(t_0)| \leq 2^m m! \max\{h,|t_1-t_0|\}^r. 
\]
\end{lemma}
\begin{proof}
See Appendix~\ref{sec:proofsA}.
\end{proof}

\section{Lower bounds}\label{sec:lower}
The subject of this section is to prove strong worst-case error bounds from below for the classes of ridge functions~$R_d^{r,p}$ in the deterministic and randomized setting. We establish the lower bounds by constructing suitable ``fooling'' ridge functions, which force any sampling algorithm to produce a large error. The general idea thereby is as follows. Suppose we find a ridge vector~$a$ such that for all~$n$ sampling points~$x_1, \ldots, x_n \in [-1,1]^d$ the inner products fulfill~$a^\tr x_i < \lambda/2$. Then the algorithm cannot distinguish any profile in~$B^{\Lip(r)}$ which is supported only on~$[\lambda/2,1]$ from the the profile that is constantly zero. A good fooling profile in this respect is the truncated power
\begin{align}\label{chap:ridge:eq:fooling_profile} 
 g_{r,\lambda}(t) = \max\{0,t - \lambda/2\}^{r}.
\end{align}

\subsection{A lower bound in the deterministic setting}

Instead of an explicit construction of fooling ridge vectors, we will derive their existence employing a standard method known in discrete mathematics as \emph{Erd\H os's probabilistic method}~\cite{erdos:1974:probabilistic}. We define a suitable finite subset of the possible ridge vectors and show that a random one of them satisfies our needs with positive probability. This, in particular, implies the existence of such a ridge vector. To control probabilities, we need a standard concentration inequality, which is known as
\emph{Hoeffding's inequality}. For a proof, see, e.g.,~\cite{rauhut/foucart:compressive_sensing}.

\begin{lemma}\label{chap:ridge:lem:hoeffding}
Let~$Z_1,\dots, Z_m$ be a sequence of independent random variables with expectation~$\mathbb{E}[X_i] = 0$ and~$|X_i| \leq B_i$ almost surely for~$i \in \{1,\dots,m\}$. Then, for all~$t > 0$,
\[ 
 \mathbb{P}\left(\sum_{i=1}^m X_i \geq t\right) \leq \exp\left(-\frac{t^2}{2\sum_{i=1}^m B_i}\right).
\] 
\end{lemma}

\noindent Using Hoeffding's inequality, we can prove the following lemma, which is the core ingredient of the probabilistic construction.

\begin{lemma}\label{chap:ridge:sec:lb:lem:inner_products_hoeffding}
Let~$0 < p \leq 1$,~$s \in \{1,\dots,d\}$, and~$n \in \N$ such that~$n < e^{s/8}$. Consider points~$z^1,\dots,z^n \in [-1,1]^d$. 
Then there exists an~$a \in \R^d$ with~$\|a\|_p = 1$ such that
\[ 
 a^\tr z^i < \|a\|_1/2 = s^{1-1/p}/2 \qquad \text{for}\quad i=1,\dots,n.
\]
\end{lemma}

\begin{proof}
Let 
$(\mathfrak a_i)_{i=1,\dots,s}$ be a sequence of i.i.d.\ random variables 
such that 
\[
 \mathbb{P}(\mathfrak a_i=1/s^{1/p})=\mathbb{P}(\mathfrak a_i=-1/s^{1/p})=1/2.  
\]
By~$\mathfrak a$ we denote the~$d$-dimensional random vector given 
by
$$\mathfrak a = (\mathfrak a_1,\dots,\mathfrak a_s,0,\dots,0).$$ 
Note that~$\|\mathfrak a\|_p = 1$ with probability one. 
For fixed~$z^i$, we have
$$\mathfrak a^\tr z^i = \sum_{j=1}^s Z_j,$$ 
where~$Z_j = \mathfrak a_j z^i_j$. The random variables~$Z_1,\dots,Z_s$ are independent, each taking values in~$[-1/s^{1/p},1/s^{1/p}]$. Since~$\mathbb{E}[Z_j]=0$ for all~$j \in [s]$, 
we obtain by Lemma~\ref{chap:ridge:lem:hoeffding} that
\[
 \mathbb{P}(\mathfrak a^\tr z^i \geq s^{1-1/p}/2) \leq e^{-s/8}.
\]
Taking a union bound, we obtain
\[
 \mathbb{P}\left(\bigcap_{i=1}^n \{ \mathfrak a^\tr z^i < s^{1-1/p}/2 \}\right) 
 \geq 1-n \, e^{-s/8}.
\]
The right hand side of the previous inequality is strictly positive if 
$n<e^{s/8}$.
Thus, for any~$n<e^{s/8}$ there exists a realization~$a$ of~$\mathfrak a$ 
such that
$$a^\tr z^i < s^{1-1/p}/2$$
for all~$1\leq i \leq n$. 
By construction of~$\mathfrak a$ we have~$\|a\|_p = 1$.
\end{proof}

Given points~$x_1,\dots,x_n \in [-1,1]^d$, Lemma~\ref{chap:ridge:sec:lb:lem:inner_products_hoeffding} guarantees the existence of an~$s$-sparse~$a \in \Sphere_p$ such that all inner products~$a^\tr x_1,\dots,a^\tr x_n$ are small. To derive a lower bound for the worst-case recovery error from this finding,  we have to take into account that the sampling points are not fixed beforehand as in Lemma~\ref{chap:ridge:sec:lb:lem:inner_products_hoeffding}, but may be chosen adaptively by the algorithm given the current input. This is possible and leads to the following lower bound. 

\begin{theorem} \label{chap:ridge:sec:lb:res:lb_cube_det}
 Let~$r>0$ and~$0 < p \leq 1$. Consider the class of ridge functions~$R_d^{r,p}$ defined in~\eqref{eq:class_compressible}.
 For any~$n\in \mathbb{N}$ with~$ n < e^{d/8}$ we have
 \begin{equation*}
 \error(n,R_d^{r,p}),L_\infty)
\geq c_r \left( \frac{1}{8 \log(2n)} \right)^{r(1/p-1)},
 \end{equation*}
where~$c_ r = 2^{-r} \,\frac{\Gamma(r+1-s)}{\Gamma(r+1)}$ and 
$s=\llceil r \rrceil$. Here,~$\Gamma$ denotes the gamma function given by~$$\Gamma(z) = \int_0^\infty x^{z-1} e^{-x} \dd x$$ for~$z > 0$. Note that we have
$c_r = 2^{-r}/r!$ for~$r \in \N$. 
\end{theorem}

\begin{proof}
Let~$S_n$ be an arbitrary deterministic, 
adaptive algorithm having a budget of~$n$ sampling points. 
Further, we denote by~$z_1,\dots,z_n\in [-1,1]^d$ those
sampling points which are successively used by~$S_n$ 
if the input is the zero function. 
Choose~$s \in \{1,\dots,d\}$ 
to be the smallest~$s$ such that
$$e^{s/8}/2 \leq n < e^{s/8}$$
and 
let~$a^* \in \R^d$ be given by Lemma~\ref{chap:ridge:sec:lb:lem:inner_products_hoeffding}
depending on the points~$z_1,\dots,z_n$.
Set~$\lambda = \|a^*\|_1$ and put~$g^*(t) = g_\lambda(t)/\|g_\lambda\|_{\Lip(r)}$, 
where~$g_\lambda$ is defined in~\eqref{chap:ridge:eq:fooling_profile}. 
The normalization assures that for the ridge 
function~$f^*(x) = g^*(a^{*\tr} x)$ we have~$ \pm f^* \in R_d^{r,p}$.

Let~$x_1, \dots, x_n\in [-1,1]^d$ 
be the sampling points 
successively chosen by~$S_n$ if the input is the ridge function~$f^*$. 
Since~$S_n$ is deterministic, 
we necessarily have~$x_1 = z_1$. 
We also have~$f^*(x_1) = 0$ by construction. It follows inductively that~$x_i = z_i$ and~$f^*(x_i) = 0$ for all~$i=1,\dots,n$. Consequently, we have~$S_n(f^*)=S_n(-f^*)=S_n(0)$ and
\begin{align*}
  \sup_{h\in R_d^{r,p}} \norm{h-S_n(h)}_{\infty} 
  &\geq  \max\left\{\norm{f^*-S_n(f^*)}_{\infty},\norm{f^*+S_n(-f^*)}_{\infty}\right\} \\
  &\geq \norm{f^*}_{\infty} = \norm{g_\lambda}_{{\rm Lip}(r)}^{-1}
   \sup_{t\in [-\norm{a^*}_1,\norm{a^*}_1]} \abs{g_\lambda(t)}\\ 
  &= \frac{2^{-r}}{\norm{g_\lambda}_{\rm Lip(r)}} \|a^*\|_1^r.
\end{align*}
Standard calculations show that 
$c_r := 2^{-r} \,\frac{\Gamma(r+1-s)}{\Gamma(r+1)} 
\leq 2^{-r}/\norm{g_\lambda}_{\rm Lip(r)}~$ with~$c_r$
independent of~$\lambda$.
Moreover,~$\|a^*\|_1 = s^{1-1/p} \geq (8 \log (2n))^{1-1/p}$. 
Since~$S_n$ was arbitrary the desired lower bound follows.
\end{proof}

\noindent The lower bound allows us to draw an immediate conclusion on the tractability of the uniform recovery problem.

\begin{corollary}\label{res:curse_det}
Let~$S \in \{1,\dots,d\}$ and $r > 0$. The $L_\infty$-recovery of ridge functions from the class~$R_d^{r,1}$ or the class~$R_d^{r,(1,S)}$ using deterministic sampling algorithms suffers from the curse of dimensionality.
\end{corollary}

\begin{proof}
Since
$$R_d^{r,1} = R_d^{r,(1,S)}, $$
the result is an immediate consequence of Theorem~\ref{chap:ridge:sec:lb:res:lb_cube_det} and the definition of the curse of dimensionality.
\end{proof}

\subsection{A lower bound for randomized algorithms}

The probabilistic constructed described in the previous section can be generalized to work in the randomized
setting, as well. We begin with the counterpart to Lemma~\ref{chap:ridge:sec:lb:lem:inner_products_hoeffding}.

\begin{lemma}\label{lem:inner_products_bakhvalov}
Let~$0<p \leq 1$ and~$0<\delta<1$. For $n \in \N$ such that~$n < (1-\delta) e^{d/8}$, let~$Z^1,\dots,Z^n$ be random vectors defined on a common probability space~$(\Omega,\mathfrak A,\prob)$ that take values in~$[-1,1]^d$. Then, for any~$s \in \{1,\dots,d\}$ that fulfills~$n < (1-\delta) e^{s/8}$, there exists an~$s$-sparse vector~$a \in \R^d$ such that
$$\|a\|_p=1, \quad \|a\|_1 = s^{1-1/p},
$$
and
\[ 
 \prob\big( \bigcap_{i=1}^n \{a^\tr Z^i < \|a\|_1/2\}\big) > \delta.
\]
\end{lemma}

\begin{proof}
Let~$\lambda = s^{1-1/p}$,~$\Omega' = \{-s^{-1/p}, s^{-1/p}\}^s \times \{0\}^{d-s}$ and let~$\prob'$ denote the uniform distribution on~$\Omega'$. The projections
$$\mathfrak a_i: \Omega' \ni a \mapsto a_i, \quad i=1,\dots,s,$$
form a sequence of i.i.d. random variables with
$$\prob'(\mathfrak a_i=-s^{-1/p})=\prob'(\mathfrak a_i = s^{-1/p}) = 1/2.$$
For the~$d$-dimensional random vector~$\mathfrak a = (\mathfrak a_1,\dots,\mathfrak a_s,0,\dots,0)$ (which is the identity on~$\Omega'$), Hoeffding's inequality yields, for all~$\omega \in \Omega$, that
\[ 
 \prob'(\mathfrak a^\tr Z^i(\omega) \geq \lambda/2) \leq e^{-s/8}.
\]
By Fubini's theorem, the same estimate holds true with respect to the product probability measure~$\widetilde{\prob} = \prob \otimes \prob'$. Namely, for the event
$$\widetilde{\Omega}_i := \{(\omega,a) \in \Omega \times \Omega' \mid a^\tr X_i(\omega) \geq \lambda/2\},$$
we have 
\[
 \widetilde{\prob}(\widetilde{\Omega}_i) = \int_{\Omega} \prob'(\mathfrak a^\tr Z^i(\omega) \geq \lambda/2) \prob(d\omega) \leq e^{-s/8}.
\]
With the convention~$\widetilde{\Omega}_0 = \Omega \times \Omega'$, we can derive from the above estimate that
\begin{align*} 
 \widetilde{\prob}\big(\bigcap_{i=1}^n \widetilde{\Omega}_i^c\big) 
 &= 1 - \sum_{i=1}^n \widetilde{\prob}\big(\widetilde{\Omega}_i \cap \bigcap_{j=0}^{i-1}\widetilde{\Omega}_j^c \big)\\
 &\geq 1 - \sum_{i=1}^n \widetilde{\prob}\big( \widetilde{\Omega}_i\big)
 \geq 1 - ne^{-s/8} > \delta,
\end{align*}
where the last estimate is due to the choice of~$s$. Applying Fubini's theorem once again and noting that~$\lambda = \|a\|_1$ for all~$a \in \Omega'$, we obtain
\begin{align*} 
 \max_{a \in \Omega'} \prob\big( \bigcap_{i=1}^n \{a^\tr Z^i < \|a\|_1/2\}\big)
 &\geq |\Omega'|^{-1}\sum_{a \in \Omega'} \prob\big( \bigcap_{i=1}^n \{a^\tr Z^i < \|a\|_1/2\}\big)\\
 &= \widetilde{\prob}\big(\bigcap_{i=1}^n \widetilde{\Omega}_i^c\big).
\end{align*}
Hence, we have
$$\max_{a \in \Omega'} \prob\big( \bigcap_{i=1}^n \{a^\tr Z^i < \|a\|_1/2\}\big) > \delta,$$
which guarantees the existence of some~$a \in \Omega'$ which has the desired properties.
\end{proof}

We come to the main result of this section, a counterpart to Theorem~\ref{chap:ridge:sec:lb:res:lb_cube_det} for the randomized setting. Up to this point, we have only considered finite probability spaces such that measurability was not an issue. This is different now that we are considering algorithms that employ random functions.

\begin{theorem} \label{chap:ridge:sec:lb:res:lb_cube_ran}
Let~$r>0$,~$0<p\leq 1$, and~$0<\delta \leq 1$. For any~$n\in \mathbb{N}$ with~$n+1< e^{d/8}/2$ we have
 \[
     \ranerror(n,R_d^{r,p}([-1,1]^d), L_\infty) 
     \geq c_r' \left(\frac{1}{8\log(4(n+1))}\right)^{r(1/p-1)},
 \]
with a constant~$c_r'$ depending only on the smoothness parameter~$r$.
 \end{theorem}

\begin{proof}
Let~$\delta = 1/2$. For~$s \in \{1,\dots,d\}$ being the smallest integer such that
$$0.5 (1-\delta)e^{s/8} \leq n+1 <  (1-\delta)e^{s/8},$$
let
$$\Omega' := \{-1/s^{1/p},1/s^{1/p}\}^s \times \{0\}^{d-s}.$$
Further, we define a ridge function
$$f_a(x) = g_\lambda(a^\tr x)/\|g_\lambda\|_{\Lip(r)}$$
for each~$a \in \Omega'$, where~$\lambda = \|a\|_1 = s^{1-1/p}$ and~$g_\lambda$ as defined in~\eqref{chap:ridge:eq:fooling_profile}. Note that~$\|f_a\|_\infty = \varepsilon_0$ for all~$a \in \Omega'$, where~$\varepsilon_0 = 2^{-r}\|a\|_1^r/\|g_\lambda\|_{\Lip(r)}$.

Let~$(\Omega,\mathfrak A,\prob)$ be the probability space and let~$S_n$ be a randomized sampling algorithm according to Definition~\ref{chap:ibc:defi:randomized_algorithm}. That is, the algorithm is given by
$$S_n(f) = \phi(f(X_1),\dots,f(X_n)),$$
where~$X_1,\dots,X_n$ are adaptively chosen,~$[-1,1]^d$-valued random variables according to Definition~\ref{chap:ibc:defi:randomized_information} and~$\phi$ is a Borel measurable random function taking values in the space of all mappings~$\R^n \to C([-1,1]^d)$.
We show in the following that there is at least one~$f_a$ such that
$$\|f_a -S_n(f_a)\|_{L_\infty} \geq \varepsilon_0/2$$ with probability at least~$\delta$. We first consider the situation that the algorithm has observed only the function value~$0$, i.e., the event
\[ 
  \Omega_{n,a} := \bigcap_{i=1}^n\{\omega \in \Omega \mid f_a(X_i(\omega)) = 0 \}
\]
has occurred. By~$\phi^0$ we denote the random function used by~$S_n$ if~$$f(X_1) = \cdots = f(X_n) = 0,$$
i.e.,~$\phi^0(\omega) = \phi(\omega)(0,\dots,0)$,~$\omega \in \Omega$. Assuming w.l.o.g. that~$\Omega'$ is ordered, we define the~$d$-dimensional random vector~$A^0$ by
\begin{align*} 
 A^0(\omega) =
 \begin{cases}
  \min\{a \in \Omega' \mid \|f_a - \phi^0(\omega)\|_{L_\infty} \leq \varepsilon_0/2 \}, & \text{if the minimum exists,}\\
  0, & \text{otherwise}.
 \end{cases}
\end{align*}
For any~$a \in \Omega'$, consider the event~$\Omega_a := \{ \omega \in \Omega: a^\tr V(\omega) < \lambda/2\}$, where~$V := \sign A^0$. Conditionally on~$\Omega_a$ we have~$\|f_a - \phi^0\| \geq \varepsilon_0/2$. Namely, for those~$\omega \in \Omega_a$ such that~$A^0(\omega)=0$, there is nothing to prove. Otherwise, on~$\Omega_a \setminus \{A^0 = 0\}$, we have
\[ 
 \|f_a - f_{A^0}\|_{L_\infty} \geq |f_a(V) - f_{A^0}(V)| = |f_{A^0}(V)| = \varepsilon_0,
\]
and thus,
\[ 
 \|f_a - \phi^0\|_{L_\infty} \geq \|f_a - f_{A^0}\|_{L_\infty} - \|f_{A^0} - \phi^0\|_{L_\infty} \geq \varepsilon_0/2.
\]

The considerations made so far show that
\begin{align*}
 \prob(\|f_a - S_n(f_a)\|_{L_\infty} \geq \varepsilon_0/2)
 \;\geq\; & \prob(\{\|f_a - S_n(f_a)\|_{L_\infty} \geq \varepsilon_0/2\} \cap  \Omega_{n,a} \cap \Omega_a)\\
 \;=\;&\prob(\Omega_{n,a} \cap \Omega_a)
\end{align*}
for any~$a \in \Omega'$. Hence, what remains is to show that~$\prob(\Omega_{n,a} \cap \Omega_a) > \delta$ for some~$a \in \Omega'$. To this end, consider the sequence of random  sampling points~$Z_1,\dots,Z_n$ which the algorithm~$S_n$ uses in case that the input function is identical to zero. Obviously, we have~$X_1(\omega) = Z_1(\omega)$. Furthermore, if~$f(X_1(\omega)) = \dots = f(X_{j-1}(\omega)) = 0$ for some~$j\geq 2$, then, it follows by induction that~$X_j(\omega) = Z_j(\omega)$. Consequently, for any~$a \in \Omega'$,
\[
 \Omega_{n,a} = \bigcap_{i=1}^n \{ f_a(Z^i) = 0 \} = \bigcap_{i=1}^n \{ a \cdot Z^i < \lambda/2\},
\]
where the last equality obviously follows from the definition of~$f_a$. Since the random variables~$Z_1, \dots, Z_n$ and~$V$ are completely independent of~$a$ and take values in~$[-1,1]^d$, we may apply Lemma~\ref{lem:inner_products_bakhvalov} with~$Z_{n+1} = V$ to obtain~$\prob(\Omega_{n,a^*} \cap \Omega_{a^*}) > \delta$ for some~$a^* \in \Omega'$. 

Now estimate
\[
 \expec[\|f_{a^*} - S_n(f_{a^*})\|_{L_\infty}^2]^{1/2}
 \geq \sqrt{\prob(\|f_a - S_n(f_a)\|_{L_\infty} \geq \varepsilon_0/2)} \varepsilon_0/2 \geq 2^{-3/2} \varepsilon_0.
\]
and since~$S_n$ is arbitrary we conclude that
\begin{align*} 
 \ranerror(n, R^{r,p}_d)
 &\geq 2^{-3/2} \varepsilon_0\\
 &\geq c_r' s^{r(1-1/p)}\\
 &=c_r'  \left(\frac{1}{8\log(4(n+1))}\right)^{r(1/p-1)}, 
\end{align*}
where~$c_r' =  2^{-3/2} c_r$ and~$c_r$ is the constant given in Theorem~\ref{chap:ridge:sec:lb:res:lb_cube_det}.
\end{proof}

\section{An algorithm for ridge functions on the cube}
\label{sec:algo}

For $r>1, 0<p \leq 1$, and $S \in \{1,\dots,d-1\}$, consider the class of ridge functions~$R_d^{r, (p,S)}$ defined in~\eqref{eq:class_approximately_sparse}. If $p=1$, we know by Corollary~\ref{res:curse_det} that the recovery of an unknown ridge function suffers from the curse of dimensionality. On the other side, if we additionally know in advance the signs of the unknown ridge vector, then the problem is polynomially tractable, which follows from~\eqref{eq:cohen_ub}. So the question that remains is how hard is it to recover an unknown ridge function if the ridge vector is approximately sparse in the sense of~\eqref{eq:approximate_sparsity} but we do not know its signs in advance.

In this section, we design an algorithm that tries to exploit this a priori knowledge by finding the signs of the largest components of the unknown ridge vector. This algorithm extends the adaptive algorithm developed in~\cite[Section~3]{cohen:capturing_ridge}. The error analysis will be done in Section~\ref{sec:error}, where we also discuss consequences for the tractability of the recovery problem.

\subsection{Recap: recovery given the signs of the ridge vector}

To better understand our idea how to compensate for dropping assumption~\eqref{eq:known_signs} it is instructive to recapitulate the basic ingredients of the adaptive algorithm described in~\cite[Section~3]{cohen:capturing_ridge}. Let~$f$ be a ridge function given by~$f(x) = g(a^\tr x)$ with unknown profile~$g \in B^{\Lip(r)}$ and unknown ridge vector~$a$ such that~$\|a\| \leq 1$ and the signs
$$\sign(a) = (\sign (a_1),\dots, \sign ( a_d))$$
are given to us in advance. First note that~$g$ and~$a$ are not uniquely determined since
\[ 
 f(x) = g(a^\tr x) = g_{\sign(a)}(\bar a^\tr x),
\]
where~$\bar a = a/\|a\|_1$ and
\[ 
 g_{\sign(a)}: [-1,1] \to \R, \quad t \mapsto f(t \sign(a)) = g(t \|a\|_1).
\]
Since~$\max_{x \in [-1,1]^d} |a \cdot x| = a \cdot \sign(a) = \|a\|_1$,
only~$g$ restricted to~$[-\|a\|_1, \|a\|_1]$ contributes to~$f$ and thus~$g_{\sign(a)}$ comprises all relevant information about~$g$. As we know the sign vector~$\sign(a)$, we can access~$g_{\sign(a)}(t)$ via~$f(t \sign(a))$ for any~$t \in [-1,1]$. Hence, we can use univariate splines to first approximate~$g_{\sign(a)}$.

To approximate the ridge direction~$\bar a$, we first search the interval~$[-1,1]$ for a point~$t^*$ such that~$|g_{\sign(a)}'(t^*)|$ is sufficiently large. Then, we consider the vector~$x^* = t^* \sign(a)$ and use first-order differences to approximate 
\[ 
 \nabla f(x^*) = g'(\|a\|_1 t^* ) \, a = g_{\sign(a)}'(t^*) \, \bar a
\]
and exploit~$\nabla f(x^*) / \|\nabla f(x^*)\|_1 = \bar a$. If there is no~$t^*$ such that~$g'_{\sign(a)}(\bar a^\tr x^*)$ 
is sufficiently large, then~$g_{\sign(a)}$ 
must be approximately constant and we can approximate~$f$ by a constant function.

\subsection{The algorithm}\label{subsec:algorithm_idea}

Not knowing~$\sign(a)$ in advance, 
it is still possible to select some~$v \in \{-1,1\}^d$ 
and search the univariate function
\begin{align}\label{eq:defgv}
 g_v: [-1,1] \to \R, \quad t \mapsto f(tv) = g(t \, a^\tr v), 
\end{align}
for a point~$t^*$ such that~$g_v'(t^*)$ is sufficiently large, i.e., $g_v'(t^*) > n_g^{-r}$.
If such a point is found, 
then one can proceed very similar as in \cite{cohen:capturing_ridge}. 
The crucial difference comes to light when we do not find~$t^*$. 
The function~$g_v$ is the rescaled restriction of the profile~$g$ 
to the interval~$[-|a \cdot v|, |a \cdot v|]$. 
If~$v \neq \sign(a)$, then~$|a \cdot v| < \|a\|_1$ 
and there is a relevant part of~$g$ which we cannot observe. 
Consequently, not finding~$t^*$ does not 
imply that~$g_{\sign(a)}$ is constant within the 
approximation accuracy.
If we manage to control the 
difference~$|\|a\|_1 - a \cdot v|$, 
however, then the regularity of~$g$ guarantees that~$g_{\sign(a)}$ 
cannot be too different from~$g_v$.

The key to find a~$v$ such that~$|a \cdot v|$ is close to~$\|a\|_1$ is the approximate sparsity of~$a$.  We argue that it suffices that~$v$ agrees with~$\sign(a)$ on the~$s$ most relevant components of~$a$. To make this latter point precise, let
$$I_{s,a} \subseteq \{1, \ldots, d\}$$
be the set of indices of the~$s$ in absolute value largest components of~$a$ (breaking ties arbitrarily). In Lemma~\ref{lem:best_r_approx}, we prove that for every vertex~$v$ from the set 
\begin{align}\label{chap:ridge:sec:cube:eq:HIT}
 \HIT_{s,a} := \Big\lbrace v \in \{-1,1\}^d: v_i = \sign(a_i) \text{ for } i \in I_{s,a} \Big\rbrace.
\end{align} 
the difference is at most
\[ 
 | \|a\|_1 - a \cdot v | \leq 2 s^{1-1/p}.
\]

How can we find a vector that is in the set $\HIT_{s,a}$? We sketch in the following how to construct a vertex set~$\mathcal V$ such that while iterating over $\mathcal V$, we are guaranteed to find an element from $\HIT_{s,a}$, see Section~\ref{sec:error} for details. It is relatively easy to establish an absolute guarantee using basic combinatorics. However, if we content ourself with finding a suitable vertex with high probability, then a vertex set $\mathcal V$ of much smaller cardinality will suffice.
Assume that we simply draw~$v$ uniformly at random from~$\{-1,1\}^d$. The event
$$\{ v \in \HIT_{s,a} \}$$
then says that we have 
guessed the~$s$ most important signs of~$a$ correctly. 
Since a random~$v$ is in~$\HIT_{s,a}$ with probability~$2^{-s}$, we can ensure that the event~$\{ v \in \HIT_{s,a} \}$ occurs at least once with probability~$1-(1-2^{-s})^{n_v}$ by taking a large enough number~$n_v$ of i.i.d.\ samples. 
It remains to choose~$s$ 
appropriately with regard to the desired approximation error. By a union bound argument, this construction can be derandomized. With high probability, we construct a set of vertices~$\mathcal{V}^{\rm det} \subseteq \{-1,1\}^d$ with
$$|\mathcal{V}^{\rm det}| \leq 2^s s \log_2(2d/s)$$
such that for all~$a \in \R^d$ with~$\|a\|_p \leq 1$ there is~$v \in \mathcal{V}^{\rm det} \cap \HIT_{s,a}$.

The considerations made so far lead to the following procedure to recover an unknown ridge function $f$ from the class~$R_d^{r,(p,S)}$, which has three parameters:
\begin{itemize}
\item a vertex set~$\mathcal V \subset \{-1,1\}^d$ that determines in which orthants we search for a large derivative of the ridge functions's profile;
\item the number of sampling points $n_g \in \N$ spent for every approximation of the profile;
\item the number of refinement steps $n_b \in \N$ used to narrow down the interval where the profile has a large derivative.
\end{itemize}

\paragraph*{Step 1:}
Until a stopping criterion is met, do the following for every~$v \in \mathcal V$. Compute the samples $f(jhv)$ where~$h = 1/n_g$ and~$j \in \{0, \pm 1, \dots, \pm n_g\}$. Let
\begin{align}\label{eq:defiL}
L_v =  \max_{-n_g \leq j < n_g} \frac{\abs{f((j+1)hv) - f(jhv)}}{h}.
\end{align}
If
\begin{align}\label{eq:L}
L_v > n_g^{-r},
\end{align}
then go to Step 2. If ultimately
\begin{align}\label{eq:noL}
 \max_{v \in \mathcal V} L_v \leq n_g^{-r}, 
\end{align}
then let 
\begin{align}
 \label{eq:fmax} f_{\max} &:= \max_{v \in \mathcal V; j=0,\pm 1,\dots,\pm n_g} f(x^{v,j}),\\
 \label{eq:fmin} f_{\min} &:= \min_{v \in \mathcal V; j=0,\pm 1,\dots,\pm n_g} f(x^{v,j}).
\end{align}
and return $\widehat f = 1/2 (f_{\min} + f_{\max})$.

\paragraph{Step 2:}
Let $[t_0,t_1]$ be the interval for which the maximum $L_v$ in Step 1 is attained. Using bisection as in~\cite{cohen:capturing_ridge}, refine this interval to an interval $[t_{\textrm{mid}} - \delta, t_{\textrm{mid}} + \delta]$ with interval length~$2\delta = h/2^{n_b}$ such that
$$
 2 \frac{|f((t_{\textrm{mid}} + \delta)v) - f((t_{\textrm{mid}} - \delta)v)|}{\delta} > n_g^{-r}.
$$

\paragraph{Step 3:}
Let $z_0 = (t_{\textrm{mid}} - \delta)v$, $z_1=(t_{\textrm{mid}} + \delta)v$, $x_0 = t_{\textrm{mid}}v$, and
$$
 x_i = t_{\textrm{mid}}v + \delta e_i, \qquad i=1,\dots,d,
$$
where $e_1,\dots,e_d$ are the canonical unit vectors. Compute the vector $\tilde{a}$ with components
\begin{align}\label{eq:approx_a_i}
 \tilde{a}_i = 2\frac{f(x_i)-f(x_0)}{f(z_1) - f(z_0)}.
\end{align}
Set $\widehat a = \tilde a/\|\tilde a\|$. This is our approximation of the ridge direction $\bar a$.

\paragraph{Step 4:}
Let $h = n_g^{-1}$. Evaluate $f$ at the points $jh \sign(\hat a)$, where $j \in \{0,\pm 1, \dots, \pm n_g \}$, which yields the samples $g_{\sign(\hat a)}(jh)$. Use these to compute the quasi-interpolant
$$\widehat g = Q_h g_{\sign(\hat a)},$$ 
which is our approximation of the profile. 

\paragraph{Step 5:}
Return the final approximation $\widehat f$ which is given by $\widehat f(x) = \widehat g(\widehat a^\tr x)$.

\paragraph*{}

We clarify the choice of the parameters in Section~\ref{sec:error}. Note that Step 1 is the crucial part where our algorithm differs essentially from the original algorithm studied in~\cite{cohen:capturing_ridge}. Step 2 corresponds to QSTEP2 in~\cite{cohen:capturing_ridge},  Steps 3 and 4 basically combine \textrm{QSTEP3} and \textrm{RSTEP2} in~\cite{cohen:capturing_ridge}.

\begin{remark}
As in~\cite{cohen:capturing_ridge}, we could in principle refine the scheme by techniques from compressed sensing to further exploit the approximate sparsity of the ridge vector. However, our subsequent analysis shows that for a ridge function~$f \in R^{r,(p,S)}([-1,1])$, the overwhelming fraction of function samples is spent for Step 1 so that sparse recovery methods have no effects in this setting in terms of tractability.
\end{remark}

\section{Error analysis}\label{sec:error}

The analysis of the algorithm described in Section~\ref{subsec:algorithm_idea} is rather lengthy and technical. Basically, we have to distinguish two scenarios:
\begin{enumerate}
 \item[(A)] Step 1 finds an interval of large deviation, i.e, Eq.~\eqref{eq:L} is fulfilled;
 \item[(B)] Step 1 does not find such an interval, i.e., Eq.~\eqref{eq:noL} holds true.
\end{enumerate}
This case distinction is analogous to the proof of~\cite[Thm. 3.2]{cohen:capturing_ridge}. In particular, if Scenario~(A) occurs, then the error analysis differs from~\cite{cohen:capturing_ridge} only at some minor technical details. The crucial difference comes to light when Scenario~(B) occurs. Then, the error analysis is much more involved than in the proof of~\cite[Thm. 3.2]{cohen:capturing_ridge} since we are not guaranteed to have sampled the complete relevant part of the profile. As we have already sketched in the previous section, choosing the vertex set~$\mathcal V$ appropriately is crucial in Scenario~(B).

\subsection{Error analysis for Scenario~(A)}\label{sec:scenarioA}

Since the error analysis is analogous to~\cite[Thm. 3.2]{cohen:capturing_ridge}, we present only the results in this section. The interested reader will find the proofs in Section~\ref{sec:proofsA} in the appendix. We begin with analyzing the recovery of the ridge direction~$\bar a = a/\|a\|_1$ in Step 3. Note that in Scenario (A) the distinction between approximate sparsity and compressibility is irrelevant. 

\begin{lemma}[Error analysis for Step 3] \label{lem: RECA}
For~$r > 1$ and~$0<p\leq1$, let~$f$ be a rige function given by~$f(x)=g(a^\tr x)$ with $g \in B^{\Lip(r)}$ and $\|a\|_p \leq 1$. Set~$\rho = \min\{r-1,1\}$. Given~$n_g \in \N$ and~$\varepsilon > 0$, choose
\begin{align}
 \label{eq: n_b} n_b := \lceil \rho^{-1} \log_2(4n_g^{r-\rho} (3+\varepsilon) \varepsilon^{-1})  \rceil
\end{align}
for the bisection performed in Step 2. Let~$v \in \{-1,1\}^d$ such that~\eqref{eq:L} holds true. Then, for the approximation 
$\widehat a$ computed by Step 3, we have
\[ 
 \|\sign(a^\tr v) \widehat a - a/\|a\|_1\|_1 \leq \varepsilon/3.
\]
\end{lemma}

\noindent Next, we combine the previous lemma with an error analysis for Step 4, which recovers the profile~$g_{\sign(\widehat{a})}$. This leads to the following recovery guarantee for any ridge function with profile in $B^{\Lip(r)}$ and ridge vector $a$ such that $\|a\|_p \leq 1$ for given $0<p\leq 1$. This completes the analysis of Scenario~(A).

\begin{theorem}\label{chap:ridge:sec:cube:res:large_L}
For~$r > 1$ and~$0<p\leq1$, let~$f$ be as in Lemma~\ref{lem: RECA}. Given~$\varepsilon > 0$, choose
\begin{align}\label{eq:n_g} 
 n_g := \lceil (10\,c_r/\varepsilon)^{1/r} \rceil,
\end{align}
and~$n_b$ as in~\eqref{eq: n_b}.
Let~$v \in \{-1,1\}^d$ such that~\eqref{eq:L} holds true. Let~$\widehat{f}$ be the ridge function given by~$\widehat{f}(x) = \widehat{g}(\widehat{a}^\tr x)$, where~$\widehat{a}$ is computed in Step 3 and~$\widehat g$ is computed in Step 4. Then, we have
\[ 
 \|f - \widehat{f}\|_\infty \leq \varepsilon.
\]
\end{theorem}

\subsection{Error analysis for Scenario~(B)}

In this section, we show that Scenario~(B) implies that the unknown ridge function~$f$ can be approximated sufficiently well by a constant function, provided the set~$\mathcal V$ has certain properties. We start with a simple observation.

\begin{lemma}\label{lem:g_v_constant}
Let~$n_g \in \N$,~$v \in \{-1,1\}^d$, and~$f \in R_d^{r,(p, S)}$. For the given ridge function~$f$, consider the function~$g_v$ as defined in~\eqref{eq:defgv}. Put
\[
 f^{v}_{\max} := \max_{i\in \{0,\pm 1,\dots, \pm n_g\}} g_{v}(ih),
 \qquad 
 f^{v}_{\min} := \min_{i \in \{0, \pm 1,\dots,\pm n_g \}} g_{v}(ih),
\]
and~$f^{v} := (f_{\min}^{{v}} + f_{\max}^{{v}})/2$. If~$L_v \leq n_g^{-r}$, where~$L_v$ is defined by~\eqref{eq:defiL}, then there is a constant~$c_r > 0$ such that
\[ 
 \| g_v - f^v \|_\infty \leq 2 c_r n_g^{-r},
\]
\end{lemma}

\begin{proof}
The proof is analogous to~\cite[Proof~of~Thm.~3.2]{cohen:capturing_ridge}). First note that 
$$|g_{{v}}(t_i) - f^{v}| \leq L \leq h^{r}$$ 
for all~$i\in\{0,\pm 1,\dots,\pm n_g\}$, where~$h = n_g^{-1}$. Moreover, by properties~\eqref{chap:ridge:eq:spline_propQ1}
and~\eqref{chap:ridge:eq:spline_propQ2} from Lemma~\ref{chap:ridge:lem:quasi-interpolation} we obtain
\begin{align*}
\begin{split}
 \|g_{{v}} - f^{v}\|_{\infty} 
 &\leq \|g_{{v}} - f^{v} + Q_h(g_{{v}}-f^{v})\|_{\infty} 
	      + \|Q_h(g_{{v}} - f^{v})\|_{\infty}\\
 &\leq c_r\left( \|g\|_{\Lip(r)} h^{r} 
	  + \max_{i\in\{0,\pm1,\dots,\pm n_g\}} |g_{{v}}(t_i) - f^{v}|\right)
 \leq 2 c_r h^{r}.
 \end{split}
\end{align*}
\end{proof}

\noindent Lemma~\ref{lem:g_v_constant} implies that for every~$v \in \mathcal V$, the profile segment given by~$g_v$ is approximately constant. We have to clarify now when this implies that~$g_{\sign(a)}$ is approximately constant, as well. A first step is to control the difference $\|a\|_1 - |a^\tr v|$.

\begin{lemma} \label{lem:best_r_approx}
Let~$0<p<1$ and~$a \in \R^d$ with~$\|a\|_p \leq 1$. For~$s \in \{1,\dots,d-1\}$, consider the set~$\HIT_{s,a}$ defined in~\eqref{chap:ridge:sec:cube:eq:HIT}. For all~$v\in \HIT_{s,a}$, we have
\[ 
 0\leq \|a\|_1 - |a^\tr v| \leq 2 s^{1-1/p}.
\]
If~$v \in \HIT_{d,a}$, then obviously~$\|a\|_1 = |a \cdot v|$.
\end{lemma}
\begin{proof}
Let~$\pi \colon \{1,\dots,d\} \to \{1,\dots,d\}$ denote a permutation which determines
the \emph{non-increasing rearrangement}, say~$a^*$, of~$a$. This means we have
$$a^*=(a_{\pi(1)},\dots,a_{\pi(d)}) \quad \text{ and } \quad a_{\pi(1)} \geq \dots \geq a_{\pi(d)}.$$
Put~$\tilde v = \sign(a^\tr v) v$. By definition of the set~$\HIT_{r,a}$, see~\eqref{chap:ridge:sec:cube:eq:HIT}, and the fact that~$
 \sigma_{s}(a) = \sum_{i=s+1}^d \abs{a^*_i}
$,
we have
\begin{align*}
 |a^\tr v | = a^\tr \tilde v
 & = \sum_{i=1}^s \abs{a^*_i} + \sum_{i=s+1}^d a^*_i \tilde v_{\pi(i)}
= \norm{a}_1 - \sigma_{s}(a) + \sum_{i=s+1}^d a^*_i \tilde v_{\pi(i)}.
\end{align*}
Hence
\begin{align*}
 0 \leq \|a\|_1 - |a  \cdot v| = \sum_{i=1}^s a^*_i (\sign(a^*_i) - \tilde v_{\pi(i)}) 
 \leq 2\sigma_{s}(a),
\end{align*}
where
$$
\sigma_s(a) := \inf\{ \|a - z\|_1: z \in \R^d \text{ is $s$-sparse} \}
$$
is the error of the best $s$-term approximation. The claim now follows from the well-known estimate~$\sigma_s(a) \le s^{1 - 1/p}$ which holds for all~$a$ with~$\|a\|_p \le 1$, see \cite[Prop. 2.3]{rauhut/foucart:compressive_sensing}.
\end{proof}

\noindent In the course of the proof of the following theorem, it will become clear that, in order to control the error~$\|g_{\sign(a)} - g_v\|_\infty$, we have to bound the quotient
\[ 
 \frac{|\|a\|_1 - a^\tr v|}{|a^\tr v|}
\]
from above. Hence, we need a lower bound on~$|a^\tr v|$. This is the reason why we have to require that the unknown ridge vector~$a$ is not only compressible, but approximately sparse. Consequently, we have to assume that the unknown ridge function~$f$ is from the class~$R^{r,(p,S)}_d$, where~$r>1$,~$0<p<1$ and~$S \in \N$ with~$S < d$. Then, we obtain the following result, which is the centerpiece of the analysis of Scenario~(B).

\begin{theorem}  \label{chap:ridge:sec:cube:res:no_large_L}
For~$r>1$,~$0<p<1$, and~$S < d$,  let~$f\in R^{r,(p,S)}_d$ with~$f(x)=g(a^\tr x)$. Given~$n_g \in \N$ and~$\mathcal V \subseteq \{-1,1\}^d$, assume that~\eqref{eq:noL} is true and let~$s$ be the largest integer~$s \in \{S,\dots,d\}$ such that
\[ 
 \mathcal V \cap \HIT_{s,a} \neq \emptyset.
\]
Define 
\[
 \widehat f(x) = \frac{f_{\min} + f_{\max}}{2},\quad x\in [-1,1]^d,
\]
where~$f_{\max}$ and~$f_{\min}$ are given by~\eqref{eq:fmax} and~\eqref{eq:fmin}.
Let~$c_r$ be the constant appearing in Lemma~\ref{chap:ridge:lem:quasi-interpolation} and~$\tilde{c}_r = (2 + 4 c_r + 2^{\llfloor r \rrfloor} \llfloor r \rrfloor!)$.
If~$s < d$, then
\[ 
 \|f -\widehat f \|_{\infty} \leq \tilde c_r \max\{(s/S)^{1-1/p}, n_g^{-1}\}^r,
\]
whereas if~$s=d$, then
\[
  \|f -\widehat f \|_{\infty}\leq 4 c_r n_g^{-r}.
\]
\end{theorem}

\begin{proof}
Let~$v\in \mathcal V$ be one of the vectors 
for which the scalar product with~$a$ is maximized, i.e.,
\[
 v := \arg\max\limits_{v' \in \mathcal V} |a^\tr v'|.
\]
W.l.o.g. we may assume~$\sign(a^\tr v) = 1$ (since otherwise we can simply replace~$v$ by~$-v$ in the following arguments).

Let
\[
 f^{v}_{\max} := \max_{i\in \{0,\pm 1,\dots, \pm n_g\}} g_{v}(t_i),
 \qquad 
 f^{v}_{\min} := \min_{i \in \{0, \pm 1,\dots,\pm n_g \}} g_{v}(t_i),
\]
and~$f^{v} := (f_{\min}^{{v}} + f_{\max}^{{v}})/2$. 
By Lemma~\ref{lem:g_v_constant}, we have that~$g_v$ is approximately constant, 
\begin{align}\label{chap:ridge:sec:cube:eq:g_v_constant}
 \|g_{{v}} - f^{v}\|_{\infty} \leq 2 c_r h^{r}.
\end{align}

Next we show that the constant~$f^v$ is close to the constant~$\widehat f$.  According to the choice of~$v$, observe that there are indices~$i_1, i_2 \in \{-n_g+1, \dots,n_g\}$ and~$\xi_1 \in [t_{i_1-1},t_{i_1}]$, 
$\xi_2 \in [t_{i_2-1},t_{i_2}]$ such that
\[ 
 g_{v}(\xi_1)  = f_{\min}, \quad g_{v}(\xi_2) = f_{\max}.
\]
Hence, by~\eqref{chap:ridge:sec:cube:eq:g_v_constant},
\[ 
 \| \widehat f - f^v\|_\infty \leq 1/2 |g_{v}(\xi_1) - f^v| + 1/2 |g_v(\xi_2) - f^v| \leq 2c_r h^r.
\]
If~$s=d$, then~$v \in \HIT_{d,a}$ and~$g_{\sign(a)} = g_{v}$ such that the statement follows.

Otherwise, if~$S < s < d$, then we have to control
$$\|g_{\sign(a)} - \widehat f\| = \|g_{v}(\|a\|_1/|a^\tr v| \cdot) - \widehat f\|,$$
with~$g_v$ now considered as a function on~$[-1/|a^\tr v|, 1/|a^\tr v|]$. For
$$|t|\|a\|_1/|a^\tr v| \leq 1,$$
we are in the interval which we have sampled, and thus as before,
\[ 
 |g_{\sign(a)}(t) - \widehat f| \leq 4c_r h^r.
\]
The crucial case is~$|t|\|a\|_1/|a^\tr v| > 1$. Now we have to extrapolate. Henceforth assume~$t > 0$ and put~$t_1 = t \|a\|_1/|a^\tr v|$ (the arguments for~$t < 0$ are completely analogous). For $m = \llfloor r \rrfloor$, let~$T_{m,1}g_{v}$ be the order-$m$ Taylor expansion of~$g_{v}$ in the point~$1$. By the triangle inequality, we have
\[ 
 |g_{v}(t_1) - \widehat f| \leq |g_{v}(t_1) - T_{m,1}g_{v}(t_1)| + |T_{m,1}g_{v}(t_1) - g_{v}(1)| + |g_v(1) - \widehat f|.
\]
By~\eqref{chap:ridge:eq:extrapolation_propQ1}, we have
\[ 
 |g_{v}(t_1) - T_{k,1}g_{v}(t_1)| \leq 2 |t_1 - 1|^r.
\]
Furthermore, since~\eqref{eq:noL} holds true, we can compute from the representation formula~\eqref{def:divided_difference} that the divided difference
$$|D_{-h}^i(g_v,1)| \leq 2^{i-1} h^{r-i+1}$$
for all~$i=1,\dots,s$. Thus, by Lemma~\ref{chap:ridge:lem:extrapolation_propQ2},
\[ 
 |T_{k,1}g_{v}(t_1) - g_{v}(1)| \leq 2^k k! \max\{h,  |t_1 - 1|\}^r.
\]
It remains to estimate~$|t_1 - 1| \leq |a^\tr v|^{-1}|\|a\|_1 - a^\tr v|$. Since
$$\mathcal V \cap \HIT_{s,a} \neq \emptyset$$
by assumption and by definition of~$v$, there is~$v' \in \HIT_{s,a}$ such that Lemma~\ref{lem:best_r_approx} gives
$$
 |\|a\|_1 - a^\tr v| \leq |\|a\|_1 - a^\tr v'| \le 2s^{1-1/p}.
$$
Consequently,~$|a^\tr v| \geq \|a\|_1 - 2 s^{1-1/p} \geq 2S^{1-1/p}$ and
\[ 
 |T_{k,t}g_{v}(t_1) - g_{v}(t)| \leq 2^k k! \max\{h, (s/S)^{1-1/p}\}^r.
\]
We conclude
\[ 
 |g_{\sign(a)}(t) - \widehat f| \leq (2+4c_r+ 2^k k!) \max\{h, (s/S)^{1-1/p}\}^r.
\]
\end{proof}

\subsection{Choice of parameters and vertex set}
We now clarify how to choose the parameters $\mathcal V$, $n_g$, and $n_b$ of the algorithm described in Section~\ref{subsec:algorithm_idea} such that, for given~$0<\varepsilon<1$, an approximation error of at most $\varepsilon$ can be guaranteed. We first consider the case in which we randomly draw vertices.

\begin{theorem} \label{chap:ridge:sec:cube:res:ran_recover}
Assume
$$
f \in R^{r,(p,S)}_d, \quad f(x) = g(a^\tr x),
$$
where~$r > 1$,~$0<p\leq1$, and~$S \in \N$ with~$S < d$.
Let~$C_r = \max\{c_r, \tilde c_r\}$
where~$c_r$ is the constant from Lemma~\ref{chap:ridge:lem:quasi-interpolation} and~$\tilde{c}_r$ is defined in Theorem~\ref{chap:ridge:sec:cube:res:no_large_L}.

Given~$0<\varepsilon<1$ and a failure probability~$0<\delta<1$, choose 
\begin{align*}
 s & := \min\{S\lceil (C_r/\varepsilon)^{1/(r(1/p-1))} \rceil, d\}, \\
 n_v &:= 2^s \lceil \log(1/\delta) \rceil,\\
 n_g & := \lceil (C_r/\varepsilon)^{1/r} \rceil,\\
 n_b& \text{ as in Lemma~\ref{lem: RECA}}.
\end{align*}
If~$n_v < 2^d$, let ~$\mathcal V = \{v_1,\dots, v_{n_v}\}$ be~$n_v$ vertices drawn independently and uniformly at random. If~$n_v = 2^d$, then let~$\mathcal V = \{-1,1\}^d$. Given these parameter choices, the approximation~$\widehat f$ computed by the algorithm described in Section~\ref{subsec:algorithm_idea} fulfills
\[ 
\mathbb{P}( \|f - \widehat f\|_{\infty} \leq \varepsilon) \geq 
\begin{cases}
 1 & n_v \geq 2^d\\
 1-\delta & n_v<2^d.
\end{cases}
\]
\end{theorem}

\begin{proof}
\emph{Case~$n_v \geq 2^d$}: We have~$\mathcal V = \{-1,1\}^d$ and
$$\mathcal V \cap \HIT_{d,a} = \{\sign(a),-\sign(a)\}.$$
If~\eqref{eq:noL} is true, then Theorem~\ref{chap:ridge:sec:cube:res:no_large_L} yields
\[ 
 \|f - \widehat f\|_\infty \leq 4 c_r n_g^{-r} \leq \varepsilon
\]
by the choice of~$n_g$. Otherwise, if~\eqref{eq:noL} is not true, then Theorem~\ref{chap:ridge:sec:cube:res:large_L} gives
\[ 
 \|f - \widehat f\|_\infty \leq 4 c_r n_g^{-r} \leq \varepsilon
\]
by the choice of~$n_g$.

\emph{Case~$n_v < 2^d$}: Consider the set of random vertices~$\mathcal V$. If~\eqref{eq:noL} is not true, then Theorem~\ref{chap:ridge:sec:cube:res:large_L} gives
\[ 
 \|f - \widehat f\|_\infty \leq 4 c_r n_g^{-r} \leq \varepsilon
\]
by the choice of~$n_g$. The fact that~$\mathcal V$ was chosen at random is irrelevant in this case.

Assume,~\eqref{eq:noL} is true. By the definition of~$\HIT_{s,a}$, see~\eqref{chap:ridge:sec:cube:eq:HIT}, it is clear that~$\mathbb{P}( v\in \HIT_{s,a}) = 2^{-s+1}$ for any~$v \in \mathcal V$. Consequently, the probability that~$\mathcal V \cap \HIT_{r} = \emptyset$ is at most~$(1-2^{-s+1})^{n_v}$. Since
$$-2/x \leq \log(1-1/x) \leq -1/x,$$
we have~$(1-2^{-s+1})^{n_v} \leq \delta$ by our choice of~$n_v$. Hence, with probability at least~$1-\delta$,~$\mathcal V \cap \HIT_{r,a} \neq \emptyset$. Then, by Theorem~\ref{chap:ridge:sec:cube:res:no_large_L} and our choice of parameters,
\[ 
 \|f - \widehat f\| \leq \tilde c_r \max\{(s/S)^{1-1/p}, n_g^{-1}\}^r \leq \varepsilon
\]
\end{proof}

\noindent If we are willing to spend a few more samples, a randomly constructed set of vertices~$\mathcal V$ will be good for all possible ridge vectors \emph{simultaneously}. It requires just a simple union bound argument to prove this. In this way, we use randomness to construct a deterministic version of the algorithm, which uses for all possible inputs~$f$ the same set of vertices~$\mathcal V$. Since~$\mathcal V$ has been randomly constructed, we only have control over the error of this deterministic algorithm with a certain probability. 

\begin{theorem} \label{chap:ridge:sec:cube:res:uni_recover}
Assume
$$
f \in R^{r,(p,S)}([-1,1]^d), \quad f(x) = g(a^\tr x),
$$
where~$r > 1$,~$0<p\leq1$, and~$S \in \N$ with~$S < d$.
Given~$0<\varepsilon<1$ and a desired failure probability~$0<\delta<1$, choose~$s$,~$n_g$, and~$n_b$ as in Theorem~\ref{chap:ridge:sec:cube:res:ran_recover}. Further, choose
\begin{align*}
 n_v &:= 2^s \lceil s \log(d/s)) + \log(1/\delta) \rceil
\end{align*}
and let~$\mathcal V$ be as in Theorem~\ref{chap:ridge:sec:cube:res:ran_recover}.
Let~$\widehat f$ be the approximation computed by the procedure introduced in Section~\ref{subsec:algorithm_idea} given the inputs~$f,\mathcal V, n_g$, and~$n_b$. Then, we have
\[ 
\mathbb{P}\bigg( \sup_{f \in R^{r,p,S}([-1,1]^d)} \|f - \widehat f\|_{\infty} \leq \varepsilon\bigg) \geq 
\begin{cases}
 1 & n_v \geq 2^d\\
 1-\delta & n_v<2^d.
\end{cases}
\]
\end{theorem}

\begin{proof}
The proof is identical to that of Theorem~\ref{chap:ridge:sec:cube:res:ran_recover}, except for one point. Before, we had to control
\[ 
 \sup_{\substack{f \in R^{r,p,S}([-1,1]^d),\\f(x) = g(a^\tr x)}} \mathbb P( \mathcal V \cap \HIT_{r,a} \neq \emptyset) = \max_{\substack{I \subseteq \{1,\dots,d\}, |I|=s\\ u \in \{-1,1\}^s}} \mathbb P( \exists v \in \mathcal V: v_I = u \vee (-v)_I = u).
\]
Now, we use a union bound argument to see that
\begin{align*}
 &\mathbb P \bigg( (\forall I \subseteq \{1,\dots,d\}, |I|=s) \; (u \in \{-1,1\}^s) \; (\exists v \in \mathcal V) :\; v_I = u \vee (-v)_I = u  \bigg)\\
 &\geq 1 - 2^s {d \choose s} \mathbb (1-2^{-s+1})^{n_v}
\end{align*}
Since~$\log{d \choose s} \leq s \log(d/s)$, our choice of~$n_v$ yields
\[ 
 1 - 2^s {d \choose s} \mathbb (1-2^{-s+1})^{n_v} \geq 1-\delta.
\] 
\end{proof}

\subsection{Upper bounds for the worst-case error}

In this section, we translate the results from the previous section into upper bounds for the worst-case recovery error. We begin with the deterministic setting.

\begin{theorem}\label{chap:ridge:sec:cube:res:ub_det}
Let~$r > 1$,~$0<p\leq 1$, and~$0<S< d$.
For constants~$c_{p,S}, C_{r,p,S} > 0$ independent of~$n$ and~$d$, we have
\[ 
  \error(n,R_d^{r,(p,S)}) \leq C_{r,p,S}
 \begin{cases}
   1 &, 1 \leq n \leq 4d,\\
   \left(\frac{1}{\log(n)}\right)^{r(1/p-1)} &, 4d \leq n \leq c_{p,S}2^d d^{1/p-1},\\
   2^{rd} \, n^{-r} &, n \geq c_{p,S} 2^d d^{1/p-1}.
 \end{cases}
\]
\end{theorem}

\begin{proof}
Fix~$\delta=1/2$. Let~$n_v,n_g,n_b$ be as in Theorem~\ref{chap:ridge:sec:cube:res:uni_recover} and put
\[ 
 n' = n'(\varepsilon) = n_v n_g + n_g + n_b + d.
\]
For every possible~$n_v$, there is a set~$\mathcal V \subseteq \{-1,1\}^d$ of cardinality~$n_v$ such that the procedure introduced in Section~\ref{subsec:algorithm_idea} yields a mapping~$S_{n'}(f)$ which achieves a recovery error
$$\|f - S_{n'}(f)\|_\infty \leq \varepsilon$$ at worst-case information cost~$n'$.

\emph{Case~$4d\leq n \leq c_{p,S} 2^d d^{1/p-1}$.} Let~$0<\varepsilon_0 < 1$ be the smallest~$\varepsilon$ such that~$$S\lceil (C_r/\varepsilon)^{\frac{1}{r(1/p-1)}}\rceil < d.$$ Then,
\[ 
 n' \geq 2^s n_g \geq \frac{1}{4} S^{1-1/p} 2^d d^{1/p-1} = c_{p,S} 2^d d^{1/p-1}.
\]
Consequently, if
$$4d \leq n \leq c_{p,S} 2^d d^{1/p-1},$$
then there is~$\varepsilon \geq \varepsilon_0$ such that~$n \leq n'(\varepsilon)$ and~$s < d$. Moreover, we may assume~$s \geq \log\log d$, since otherwise~$n' \leq 4d$. Now, since
\begin{align*}
 n_v &\leq 2^{s+2} s \log(d/s) \leq 2^{3s+2} \leq 2 \cdot 16^{S(C_r/\varepsilon)^{\frac{1}{r(1/p-1)}}}\\
 n_g &\leq 2 (C_r/\varepsilon)^{1/r} \leq 2^{1+1/r (C_r/\varepsilon)^{\frac{1}{r(1/p-1)}}},
\end{align*}
we have
\[ 
 n'-d \leq 4n_v n_g \leq 16 \cdot 2^{(4S+1/r) (C_r/\varepsilon)^{\frac{1}{r(1/p-1)}}}. 
\]
Using the assumption~$n \geq 4d$, we find a constant~$c > 1$ such that
\[ 
 \error(n,R_d^{r,(p,S)}) \leq \varepsilon \leq c C_r (4S+1/r)^{1/p-1} \left(\frac{1}{\log(n)}\right)^{r(1/p-1)}.
\]

\emph{Case~$n \geq c_{p,S} 2^d d^{1/p-1}$.} Choose~$0<\varepsilon < \varepsilon_0$ such that~$n \leq n'(\varepsilon)$. Now
\[ 
 n'- d \leq 4 n_v n_g \leq 8 \cdot 2^d (C_r/\varepsilon)^{1/r} 
\]
and thus
\begin{align}
 \label{chap:ridge:sec:cube:eq:ub_ran_asymp} 
 \error(n,R_d^{r,(p,S)}) \leq \varepsilon \leq C_r 16^r 2^d n^{-r}.
\end{align}

\emph{Case~$1 \leq n \leq 4d$.} The trivial algorithm gives~$$\error(n,R_d^{r,(p,S)}) \leq \sup_{f \in R^{r,(p,S)}([-1,1]^d)} \|f\|_\infty = 1.$$
\end{proof}

\begin{corollary}
Let $0<p<1$, $S \in \{1,\dots,d-1\}$, and
\[ 
 r > \frac{1}{1/p-1}.
\]
Then the $L_\infty$-recovery of an unknown ridge function from the class $R^{r,(p,S)}_d$ is at least weakly tractable.
\end{corollary}
\begin{proof}
Let $C_{r,p,S}$ and $c_{p,S}$ be the constants defined in Theorem~\ref{chap:ridge:sec:cube:res:ub_det} and put
$$
\varepsilon_1 = C_{r,p,S} \left(\frac{1}{\log(4d)}\right)^{r(1/p-1)}.
$$
Then, it follows from Theorem~\ref{chap:ridge:sec:cube:res:ub_det} that there are constants $C_0$ and $C_1$ that are independent of $\varepsilon$ and $d$ such that
\begin{align*}
 \log n(\varepsilon, R_d^{r,(p,S)}) \leq C_0 + C_1
 \begin{cases}
 \log(d)                               &, \varepsilon_1 \leq \varepsilon \leq 1,\\
 (1/\varepsilon)^{\frac{1}{r(1/p-1)}}  &, \varepsilon < \varepsilon_1.\\ 
 \end{cases} 
\end{align*}
Put $x=1/\varepsilon+d$. Then, it follows that
\[ 
 \log n(\varepsilon, R_d^{r,(p,S)}) \leq C_0 + C_1 \log(x) x^{\frac{1}{r(1/p-1)}}
\]
and $\lim_{x \to \infty} x^{-1} \log n(\varepsilon, R_d^{r,(p,S)}) = 0$. By definition of weak tractability, the desired result follows.
\end{proof}

For completeness, let us also consider the randomized version of the algorithm described in Section~\ref{subsec:algorithm_idea}. Although the randomized version is less costly than its deterministic counterpart, the following result shows that we basically have the same upper bounds as in the deterministic setting.

\begin{theorem}\label{chap:ridge:sec:cube:res:ub_ran}
Let~$r > 1$,~$0<p\leq 1$, and~$0<S< d$.
For~$c_{p,S}$ as in Theorem~\ref{chap:ridge:sec:cube:res:ub_det} and a constant~$C_{r,p,S}>0$ independent of~$n$ and~$d$, we have
\[ 
 \ranerror(n,R_d^{r,(p,S)}) \leq C_{r,p,S}
 \begin{cases}
   1 &, 1 \leq n \leq 2d,\\
   \left(\frac{1}{\log(n)}\right)^{r(1/p-1)} &, 2d \leq n \leq c_{p,S}2^d d^{1/p-1},\\
   2^{rd} \, n^{-r} &, n \geq c_{p,S} 2^d d^{1/p-1}.
 \end{cases}
\]
\end{theorem}
\begin{proof}
Let~$n_v, n_b, n_b$, and~$\mathcal V$ as in Theorem~\ref{chap:ridge:sec:cube:res:ran_recover}. With these choices, the procedure introduce in Section~\ref{subsec:algorithm_idea} yields a mapping
$
 S_{n'}(f)
$
with information cost
$$n' = n'(\varepsilon) = n_v n_g + n_g + n_b + d.$$

\emph{Case~$n > c_{p,S} 2^d d^{1/p-1}$}. We find~$0<\varepsilon\leq 1$ such that~$n \leq n'$ and~$s = d$. Then~$S_{n'}$ is deterministic and the argumentation is identical to the proof of Theorem~\ref{chap:ridge:sec:cube:res:uni_recover}; by~\eqref{chap:ridge:sec:cube:eq:ub_ran_asymp}, we obtain
\[ 
 \ranerror(n,R_d^{r,(p,S)}) \leq C_{r,p,S} 2^{rd} n^{-r}.
\]

\emph{Case~$2d \leq n \leq c_{p,S} 2^d d^{1/p-1}$}. Choose~$0<\varepsilon\leq 1$ such that with the choice~$\delta = \varepsilon^2$, we have~$n \leq n'$ and~$s < d$.
By Theorem~\ref{chap:ridge:sec:cube:res:ran_recover}, we have
$$
 \mathbb{P}( \|f-S_{n'}(f)\|_\infty >\varepsilon) \leq \varepsilon^2,
$$
which leads, in combination with~$\|f-\hat S_{n'}(f)\|_\infty \leq 2$ a.s., to the estimate
\begin{align*}
 &\ranerror(n,R_d^{r,(p,S)}) \leq \sqrt{\mathbb{E}\|f - S_{n'}(f)\|_{\infty}^2}\\
 &= \sqrt{\int_{\{\|f-S_{n'}(f)\|_\infty > \varepsilon\}} \|f-S_{n'}(f)\|_\infty^2
  + \int_{\{\|f-S_{n'}(f)\|_\infty \leq \varepsilon\}} \|f-S_{n'}(f)\|_\infty^2}\\
 &\leq \sqrt{5} \varepsilon.
\end{align*}
Now, since
\begin{align*}
 n_v &\leq 2^{s+2} \log(1/\varepsilon) \leq 2^{2s+2} \leq 2 \cdot 8^{S(C_r/\varepsilon)^{\frac{1}{r(1/p-1)}}},\\
 n_g &\leq 2 (C_r/\varepsilon)^{1/r} \leq 2^{1+1/r (C_r/\varepsilon)^{\frac{1}{r(1/p-1)}}},
\end{align*}
we have
\[ 
 n'-d \leq 4n_v n_g \leq 16 \cdot 2^{(3S+1/r) (C_r/\varepsilon)^{\frac{1}{r(1/p-1)}}}. 
\]
Using the assumption~$n \geq 2d$, we find a constant~$c > 1$ such that
\[ 
 \ranerror(n,R_d^{r,(p,S)}) \leq \sqrt{5}\varepsilon \leq \sqrt{5} c C_r (3S+1/r)^{1/p-1} \left(\frac{1}{\log(n)}\right)^{r(1/p-1)}.
\]
\emph{Case~$1 \leq n \leq 4d$}. the trivial algorithm gives
$$\ranerror(n,R_d^{r,(p,S)}) \leq \sup_{f \in R^{r,p,S}([-1,1]^d)} \|f\|_\infty = 1.$$ 
\end{proof}

\begin{remark}\label{chap:ridge:sec:cube:rem:p=1}
In the case~$p=1$, all results hold still true if we replace the class~$R_d^{r,(S,1)}$ by~$R_d^{r,1}$ since we only have to consider the case~$s=d$ in Theorem~\ref{chap:ridge:sec:cube:res:no_large_L} then. We do not now whether the obtained upper bounds are optimal when~$0<p<1$.
\end{remark}

\section{Related Work}\label{sec:related_work}

There is a vast body of literature that is concerned with ridge functions. Since various mathematical communities have contributed to the research on ridge functions, we find it of value to close this work with a broader overview of related research. This overview does by no means claim to be exhaustive. 

\subsection{Further work on uniform recovery}
An alternative to the component-wise positivity of the ridge vector~\eqref{eq:known_signs} that also guarantees polynomial tractability is to assume that 
\begin{align}\label{eq:assump_vybiral}
 g'(0) > \kappa
\end{align}
for some given $\kappa > 0$.
This assumption works both for ridge functions defined on the hypercube and for ridge functions defined on the Euclidean ball
$$B_2^d = \{x\in \R^d: \|x\|_2 \leq 1\},$$
whereas~\eqref{eq:known_signs} does not lead to a polynomially tractable problem for ridge functions defined on the hypercube, see~\cite{mayer/ullrich/vybiral:ridge_sampling}. Assumption~\eqref{eq:assump_vybiral} has been studied in \cite{fornasier/schnass/vybiral:multi-ridge_functions,kolleck/vybiral:ridge_approximation,mayer/ullrich/vybiral:ridge_sampling}, where~\cite{kolleck/vybiral:ridge_approximation} studies also the effect of noisy measurements. In~\cite{mayer/ullrich/vybiral:ridge_sampling}, it has been shown that recovery of ridge functions defined on the Euclidean ball in general suffers from the curse of dimensionality. This finding is based on novel two-sided estimates that reduce the decay behavior of the worst-case recovery error to the decay behavior of entropy numbers of~$\ell_p^d$-balls. 
whereas~$0<p<2$ implies weak tractability for sufficiently large~$\alpha > 0$. 

There is an obvious generalization of the ridge function model, namely functions of the form
\[ 
 f(x) = g(Ax), \quad g: \R^m \to \R, \quad A \in \R^{m \times d},
\]
where~$m$ is supposed to be much smaller than the ambient dimension~$d$. In~\cite{pinkus:ridge_functions}, such functions are called \emph{generalized ridge functions}. Following the ideas of~\cite{buhmann/pinkus:1999:identifying_ridge}, the paper~\cite{fornasier/schnass/vybiral:multi-ridge_functions} develops an efficient algorithm for the recovery of generalized ridge functions defined on Euclidean balls, provided the function~$g$ fulfills certain integral conditions and the rows of the matrix are compressible. In the case~$m=1$, the integral conditions are fulfilled, e.g., if we assume~\eqref{eq:assump_vybiral}. 
Instead of compressibility assumptions on the rows of~$A$, the work~\cite{tyagi/cevher:2014:low-rank} assumes that the matrix~$A$ is a low-rank tensor and obtains an algorithms that requires only polynomially many function samples.

A rank-1 tensor is a multivariate function of the form $f(x_1,\dots,x_d) = \prod_{j=1}^d f(x_j)$. The recent works~\cite{bachmayr/dahmen/devore/grasedyck:2014:rank-1,novak/rudolf:2016:rank-1} study efficient methods and tractability aspects. The proof techniques show interesting resemblances to the techniques used in the context of ridge functions.

\subsection{Ridge functions in semi-parametric statistics}
\label{chap:ridge:subsec:semi-parametric_statistics}
The phenomenon ``curse of dimensionality'' 
is also known in statistics. 
In the \emph{regression problem}, 
one has stochastically independent observations 
$$(X^{(1)},Y_1),\dots,(X^{(n)},Y_n),$$ 
which are assumed to be related by
\[ 
 Y_i = f(X^{(i)}) + \epsilon_i, \quad i=1,\dots,n,
\]
where the~$\epsilon_i$ are noise terms. 
The goal is to derive from these observations a reconstruction~$\hat f$ of the unknown function 
$f$ such that the least squares 
error~$\|f-\hat f\|_2$ is small. In this context, 
curse of dimensionality refers to 
the fact that the random sampling 
points~$X^{(1)},\dots,X^{(n)}$ 
are sparsely scattered when they take values in 
high-dimensional metric spaces. 
This has the unpleasant consequence that standard nonparametric regression 
techniques such as kernel estimation, nearest-neighbor, 
and spline smoothing work poorly 
in high dimensions since they are based on local averaging. 

It has been a prominent idea in statistics to allow only specific functional dependencies in models to mitigate the burden of high-dimensionality. In this way, one seeks to
find a compromise between linear models, that scale rather well with the dimension, and fully nonparametric
models, which face the issues mentioned above in high-dimensional settings. \emph{Projection pursuit regression (PPR)}~\cite{friedman/stuetzle:project_pursuit_regression,huber:projection_pursuit}
is one possible semiparametric approach used 
since the early 1980s to face the problem of sparsely scattered data. 
The key assumption is that the unknown regression 
surface~$f$ can be approximated well by a \emph{sum of ridge functions},
i.e.
\begin{align}\label{chap:ridge:eq:ppr} 
 f(x) \approx \sum_{j=1}^m g_j(a_j^\tr x)
\end{align}
with~$a_j \in \R^d$ and univariate functions~$g_j$. This can be interpreted as a non-linear generalization of \emph{principal component analysis (PCA)}~\cite{hastie/tibshirani/friedman:2009:elements}. A widely used simplification of~\eqref{chap:ridge:eq:ppr} are \emph{additive models}~\cite{hastie/tibshirani/friedman:2009:elements,raskutti/wainwright/yu:2012:additive}, where the~$a_j$ are assumed to be coordinate directions.
  
For~$m=1$ in~\eqref{chap:ridge:eq:ppr}, a closely related semiparametric model is popular in econometrics under the name \emph{single-index model}~\cite{haerdle/mueller/sperlich/werwatz:nonparametric_and_semiparametric_models,ichimura:estimation_of_single_index_models}; it assumes that~$f$ is a ridge function,
\[ 
 f(x) = g(a^\tr x).
\]
These simple ridge-based regression models have been successfully applied to high-dimensional real-world data, for instance, to identify the variables that influence income~\cite{cui/haerdle/zhu:the_efm_approach}, the severity of side impact accidents~\cite{haerdle/hall/ichimura:optimal_smoothing}, or air pollution~\cite{friedman/stuetzle:project_pursuit_regression}. As a particular family of estimation methods, we mention \emph{average gradient estimation (ADE)}~\cite{hristache/juditsky/spokoiny:direct_estimation,powell/stock/stoker:1989:ade,stoker:1986:ade}, which also rests on the idea to exploit~$\nabla f(x) = g(a^\tr x) a$. Concerning theoretical error bounds, root-$n$-consistency for various estimation methods has been shown, assuming that~$g$ is two-times differentiable and~$\|a\|_2=1$; see, e.g.,
\cite{cui/haerdle/zhu:the_efm_approach,hall:on_projection_pursuit_regression,haerdle/hall/ichimura:optimal_smoothing,hristache/juditsky/spokoiny:direct_estimation}. If a method is root-n-consistent, then this implies
\begin{align}\label{chap:ridge:eq:root-n-consistency} 
 \expec \|f - \widehat f\|_2 = \mathcal O(n^{-1/2}),
\end{align}
where~$f$ is the unknown ridge function,~$\widehat f$ the computed estimate and~$n$ the number of samples used. There is also a work that proves asymptotically optimal minimax bounds~\cite{golubev:1992:ridge}. It seems that all the afore-mentioned results are only of asymptotic nature and
hide constants which potentially depend on the dimension~$d$. We further note that the decay rate in~\eqref{chap:ridge:eq:root-n-consistency} mainly reflects the assumptions that have been made for the noise of the samples.

\subsection{One-bit compressed sensing}

In $1$-bit compressed sensing, the aim is to recover a compressible signal $a \in \R^d$ from~$1$-bit measurements $y_i = \sign(a^\tr x_i)$, $i=1,\dots,n$, given that $\expec y_i = g(a^\tr x)$ for some unknown, univariate $g: \R \to [-1,1]$ such that
\begin{align}\label{eq:1bit}
	\expec[g(X)X] = \lambda > 0
\end{align} for a standard normal random variable $X$, see~\cite{plan/vershynin:2013:1bit} and the references there. Note that the goal here is only to recover the vector $a$ and not the non-linearity $g$. Further, note the similarity between~\eqref{eq:1bit} and the integral condition discussed in~\cite{fornasier/schnass/vybiral:multi-ridge_functions}. In particular, it is clear that~\eqref{eq:1bit} is fulfilled if $g$ is a continuous function with $g(0) > \kappa > 0$.

\subsection{Ridge functions as atoms for approximation}

The afore-mentioned PPR provides an example for approximating an unknown function \emph{by} a sum of ridge functions. The recent monograph~\cite{pinkus:ridge_functions} gives a detailed overview of what is known about approximation by sums of ridge functions. This includes, among other aspects, uniqueness of representation, density properties (i.e., what functions can be approximated by sums of ridge functions), degree of approximation, best approximation, and greedy methods. We also recommend the older work~\cite{pinkus:1997:approximating-by-ridge} by the same author, which is a well-written introduction to this topic. 

Closely related to PPR in spirit and in terms of algorithmic approaches are neural network models~\cite{anthony/bartlett:1999:nn}. For instance, in \emph{single hidden-layer feedforward networks}, the given data is fitted to a function of the form
\[ 
 f(x) = \sum_{i=1}^m \beta_i \sigma(a_i^\tr x + b_i).
\]
In contrast to projection pursuit regression, the univariate~$\sigma$, which is called \emph{activation function}, is chosen in advance.  Approximation-theoretical properties of these models and also the more general \emph{multilayer feedforward perceptrons (MLP)} have been surveyed in~\cite{pinkus:1999:mlp}. A new approximation-theoretical approach towards neural networks models---and more generally, learning based on dictionaries---has been established by~\cite{candes:1999:nn}. This work introduces an analogon to the well-known Fourier and wavelet transforms based on ridge functions, the so-called \emph{ridgelet transform}. This transform provides representations with frame properties that are particularly suited to represent functions with singularities along hypersurfaces. Further statistical properties of ridgelets have been studied in~\cite{candes:2003:estimating}. The paper~\cite{donoho:2000:orthonormal_ridgelets} constructs an orthonormal basis based on ridgelets.

\begin{appendix}

\section{Further proofs}\label{sec:proofsA}

\paragraph{Proof of Lemma~\ref{chap:ridge:lem:extrapolation_propQ2}.}

There is~$\xi_m \in [1-sh,1]$ such that~$D_{-h}^m(g,1) = g^{(m)}(\xi_m)$. Further, for~$\beta = r - s$, the derivative~$g^{(m)}$ is H\"older-continuous with
$$|g^{(m)}| \leq |a^\tr v|^m |g^{(m)}|_\beta \leq |a^\tr v|^m.$$
Hence, for all~$\xi \in [1-sh,1]$ we obtain
\begin{align*} 
 |g^{(m)}(\xi)| &
 \leq |g^{(m)}(\xi_m)|+|\xi - \xi_m|^\beta\\
 & \leq \min\{|a^\tr v|^m,2^{m-1} h^{\beta}\} + |a^\tr v|^m (mh)^\beta = (2^{m-1} + s^\beta) h^{\beta} = C_m h^\beta.
\end{align*}

Considering the derivative~$g^{(m-1)}$, there is~$\xi_{m-1} \in [1-(m-1)h,1]$ such that~$$D_{-h}^{m-1}(g,1) = g^{(m-1)}(\xi_{m-1}).$$ By the mean value theorem, there is for all~$\xi \in [1-(m-1)h,1]$ a
$$\xi_m' \in [1-(m-1)h,1]$$
such that
\begin{align*} 
 g^{(m-1)}(\xi) = g^{(m-1)}(\xi_{m-1}) + g^{(m)}(\xi_m')(\xi - \xi_{m-1}).
\end{align*}
By assumption and the previous considerations we conclude
\[
|g^{(m-1)}(\xi)| \leq 2^{m-2}h^{\beta+1} + C_m (m-1) h^{\beta+1} = C_{m-1}h^{\beta+1},
\]
where~$C_{m-1} = 2^{m-2} + 2^{m-1}(m-1) + s^\beta(s-1)$.

Iteratively repeating this argument for the remaining derivatives, we obtain
\[ 
 |g^{(i)}(\xi)| \leq C_i h^{r-i}, \quad \xi \in [1-ih,1],
\]
with~$C_i = \sum_{j=i}^m 2^{j-1} \prod_{l=i}^{j-1} l + \prod_{l=i}^{s-1} s^\beta$. It is easy to see that~$C_i \leq 2^k s!$.

\paragraph{Proof of Lemma~\ref{lem: RECA}.}
We can assume that~$a^\tr v \not=0$, otherwise the profile segment~$g_v$ given by~\eqref{eq:defgv} is constant on~$[-1,1]$ and consequently, there is nothing to prove. Let~$[t_{\rm mid}-\delta, t_{\rm mid}+\delta]$ be the refined interval computed in  Step 2 using~$n_b$ samples. Recall that~$\widehat a = \widetilde{a}/\|\widetilde{a}\|_1$, where the~$i$th coordinate of~$\widetilde{a}$ is given by~\eqref{eq:approx_a_i} for~$i=1,\dots,d$. By the fact that
$$\sign(\;\norm{a}_1^{-1}/(v^\tr a)) = \sign(1/(v^\tr a))$$
and~\cite[Lemma~3.1]{kolleck/vybiral:ridge_approximation}, we have
\begin{equation} \label{eq: KV14_lemma}
 \| \sign(1/(v^\tr a))\, \widehat a - a/\|a\|_1 \|_1 
 \leq 2\; \frac{\|\widetilde a - a/(a^\tr v)\|_1}{\|\widetilde a\|_1}.
\end{equation}

Let us prove an upper bound for the right-hand side in~\eqref{eq: KV14_lemma}. Extending the definition in~\eqref{eq:defgv}, let
\[ 
 g_v: [-|a^\tr v|^{-1}, |a^\tr v|^{-1}] \to \R, \quad t \mapsto g(t a^\tr v)
\]
denote the stretched profile of which Step 3 observed function values. By the mean value theorem, there is a real number~$\xi_0$ satisfying
$$|\xi_0 - t_{\rm mid}| \leq \delta,$$
and real numbers~$\xi_i$ for~$i\in \{1,\dots,d\}$ satisfying~$|\xi_i - t_{\rm mid}| \leq \delta |a_i|/|a^\tr v|$ such that 
\begin{align*}
 g_v'(\xi_0 ) & = \frac{g_v(t_{\rm mid} +\delta)-g_v(t_{\rm mid}-\delta)}{2\delta}, \\
 g_v'(\xi_i) & = \frac{g_v(t_{\rm mid}+\frac{\delta a_i }{a^\tr v})-g_v(t_{\rm mid})}{\delta}^\tr \frac{a^\tr v}{a_i}.
\end{align*}
This implies
\[
 \widetilde a_i = \frac{a_i}{a^\tr v}^\tr \frac{g'_v(\xi_i)}{g'_v(\xi_0)} 
 = \frac{a_i}{a^\tr v} \left(1 + \frac{g_v'(\xi_i) -  g_v'(\xi_0)}{g_v'(\xi_0)}\right).
\]
Hence
\[ 
 \frac{\|\widetilde a - a/(v^\tr a)\|_1}{\|\widetilde a\|_1} 
 = \frac{\sum_{i=1}^d |g_v'(\xi_i) - g_v'(\xi_0)| |a_i|}{\sum_{i=1}^d |g_v'(\xi_i)| |a_i|}. 
\]
Now, since~$g_v'$ is H\"older continuous on~$[-|a^\tr v|^{-1},|a^\tr v|^{-1}]$ with exponent~$\rho$, we obtain by~\eqref{chap:fs:eq:hoelder_defi} that 
\begin{align*}
 |g_v'(\xi_i) - g_v'(\xi_0)| &\leq 2 |g_v'|_{\rho} \min\{1,| \xi_i - \xi_0|\}^{\rho}\\
 &\leq 2 \|g\|_{\Lip(r)} |a^\tr v| \left(| \xi_i - t_{\rm mid}|^{\rho} + |\xi_0 -  t_{\rm mid}|^{\rho} \right)\\
 &\leq 2 \delta^{\rho}  \|g\|_{\Lip(r)} \left( |a_i|^{\rho} |a^\tr v|^{1-\rho} + |a^\tr v| \right).
\end{align*}
By~$|v^\tr a| \leq \|v\|_{\infty} \|a\|_1$ and~$|a_i| \leq \|a\|_1$, it
follows that
\[ 
|g_v'(\xi_i) - g_v'(\xi_0)| \leq 4 \|g\|_{\Lip(r)} \|a\|_1 \delta^{\beta} \leq 4 \delta^{\rho}.
\]
Using~$|g_v'(\xi_i)| \geq |g_v'(\xi_0)| - |g_v'(\xi_i) - g_v'(\xi_0)| \geq L - 4 \delta^\rho$, 
we obtain
\[ 
 \frac{\|\widetilde a - a/v^\tr a\|_1}{\|\widetilde a\|_1} 
 \leq \frac{ \delta^{\rho}}{L/4 - \delta^{\rho}}.
\]
The choice of~$n_b$ guarantees that
\begin{equation} \label{eq: bisection}
\delta = 2^{-n_b} \abs{I_0} = 2^{-n_b}/n_g\leq \left(\frac{L\epsilon}{4(6+\varepsilon)}\right)^{1/\rho}.
\end{equation}
which in turn yields
\[ 
 \frac{ \delta^{\rho}}{L/4 - \delta^{\rho}} \leq \varepsilon/3.
\]
This proves the statement of this lemma.

\paragraph{Proof of Theorem~\ref{chap:ridge:sec:cube:res:large_L}.}
Let~$\gamma := \sign(v^\tr a)$. Recall that
$$f(x) = g(a^\tr x) = g_{\sign(a)}(\bar a^\tr x),$$
where~$\bar a = a/\|a\|_1$. Let~$Q_h$ denote a quasi-interpolant as introduced in Section~\ref{sec:def}. For any~$x \in [-1,1]^d$, the approximation error can be decomposed into three components,
\begin{align*}
 |\widehat f(x) - f(x)| & = 
 |( Q_h g_{\sign(\widehat{a})})( \widehat{a}^\tr x)
 - g_{\sign(a)}(\bar a^\tr x)|\\
 &\leq
 | ( Q_h g_{\sign(\widehat{a})})( \widehat{a}^\tr x) - g_{\sign(\widehat a)}(\widehat a^\tr x)|\\ 
 &\qquad\qquad+ |g_{\sign(\widehat a)}(\widehat a^\tr x) - g_{\sign(a)}(\gamma \widehat a^\tr x)| \\
 &\qquad\qquad+ |g_{\sign(a)}(\gamma \widehat a^\tr x) - g_{\sign(a)}(\bar a^\tr x)|.
\end{align*}
The first part is because we can only approximate~$g_{\sign(\widehat a)}$, 
the second component is due to the uncertainty regarding the orthant 
(the signs of the ridge vector), and the third one is due to the uncertainty regarding the ridge vector.
By Lemma~\ref{chap:ridge:lem:quasi-interpolation}, the choice of~$n_g$ gives
\[
 | ( Q_h g_{\sign(\widehat{a})})( \widehat{a}^\tr x) - g_{\sign(\widehat a)}(\widehat a^\tr x)| 
 \leq \|Q_h g_{\sign(\widehat{a})} - g_{\sign(\widehat a)}\|_{\infty} \leq c_r n_g^{-r} \leq \varepsilon/3.
\]
To treat the second term we need some preliminary calculations. Namely, as in
\cite[Eq. (3.10)]{kolleck/vybiral:ridge_approximation} we have
\begin{align*}
  \abs{\bar a^\tr (\sign(\gamma \widehat a)-\sign(\bar a))} 
= & \abs{\,\norm{\bar a}_1 - \norm{\widehat{a}}_1 - (\bar a-\gamma \widehat a )^\tr (\sign(\gamma \widehat a))} \\
\leq & \norm{\bar a-\gamma \widehat a}_1 \norm{\sign(\gamma \widehat a)}_\infty \leq \norm{\bar a-\gamma \widehat a}_1,
\end{align*} 
since~$\|\widehat a\|_1 = \|\bar a\|_1 = 1$. 
Then, for the second term we obtain
\begin{align*} 
   |g_{\sign(\widehat a)}(\widehat a^\tr x) - g_{\sign(a)}(\gamma\widehat a^\tr x)|  
 &= |g\left( ( a^\tr \sign(\widehat a) )\,(  \widehat a^\tr x ) \right) 
	  - g\left( \|a\|_1 ( \gamma \widehat a^\tr x )\right)|\\
 &\leq \|g\|_{\Lip(r)}\, |( \gamma \widehat a^\tr x ) |\,|  a^\tr (\sign(\gamma \widehat a) - \sign(a)) |\\
 &\leq  \|a\|_1 \big\|\gamma \widehat a - a/\|a\|_1 \big\|_1\\
 &\leq \big\|\gamma \widehat a - a/\|a\|_1\big\|_1
\end{align*}
and for the third term we have
\begin{align*}
     |g_{\sign(\gamma a)}(\widehat a^\tr x) - g(a^\tr x)|
 & = |g(\gamma \|a\|_1 \widehat a^\tr x) -g(a^\tr x)| \\
 & \leq  \|g\|_{\Lip(r)} \|a\|_1 \big\|\gamma \widehat a - a/\|a\|_1\big\|_1\\ 
 & \leq \big\|\gamma \widehat a - a/\|a\|_1\big\|_1.
\end{align*}
By Lemma~\ref{lem: RECA}, we have~$\big\|\gamma \widehat a - a/\|a\|_1\big\|_1 \leq \varepsilon/3$, which proves the statement.

\end{appendix}

\bibliographystyle{plain}
\bibliography{literature}

\end{document}